\documentclass[11pt,reqno]{amsart}
\usepackage[utf8]{inputenc}
\usepackage[T1]{fontenc}
\usepackage[initials,nobysame,shortalphabetic]{amsrefs}
\usepackage{enumerate}
\usepackage{mathtools}
\usepackage{xcolor}

\usepackage{geometry}
\numberwithin{equation}{section}

\theoremstyle{plain}
\newtheorem{theorem}{Theorem}[section]
\newtheorem{proposition}[theorem]{Proposition}
\newtheorem{propdef}[theorem]{Proposition-Definition}
\newtheorem{lemma}[theorem]{Lemma}

\theoremstyle{definition}
\newtheorem{definition}[theorem]{Definition}
\theoremstyle{remark}
\newtheorem{remark}[theorem]{Remark}
\newtheorem{example}[theorem]{Example}

\newcommand{\R}{\mathbb{R}}
\newcommand{\C}{\mathbb{C}}

\newcommand{\Z}{\mathbb{Z}}

\newcommand{\cover}{\mathfrak{U}}
\newcommand{\currs}{\mathcal{C}}
\newcommand{\forms}{\mathcal{E}}
\newcommand{\holo}{\mathcal{O}}

\newcommand{\opens}{\mathcal{U}}

\def\Homs{\mathop{\mathcal{H}\!\mathit{om}}\nolimits}
\def\Exts{\mathop{\mathcal{E}\!\mathit{xt}}\nolimits}

\renewcommand{\dbar}{\bar{\partial}}

\DeclareMathOperator{\codim}{codim}
\DeclareMathOperator{\End}{End}
\DeclareMathOperator{\Ext}{Ext}
\DeclareMathOperator{\Hom}{Hom}
\DeclareMathOperator{\id}{id}
\DeclareMathOperator{\im}{im}
\DeclareMathOperator{\supp}{supp}
\DeclareMathOperator{\coker}{coker}
\DeclareMathOperator{\rank}{rank}
\DeclareMathOperator{\depth}{depth}

\newcommand{\defeq}{\vcentcolon=}

\title[Different representations of Ext and residue currents]{An explicit isomorphism of different representations of the Ext functor using residue currents}
\author{Jimmy Johansson}
\author{Richard Lärkäng}
\address{Department of Mathematical Sciences,
Chalmers University of Technology and the University of Gothenburg,
Gothenburg SE-412 96, Sweden}
\email{jimmjoha@gmail.com, larkang@chalmers.se}
\date{October 7, 2024}

\begin{document}
\begin{abstract}
Let $\mathcal{F}$ be a coherent $\holo_X$-module over a complex manifold $X$, and let $G$ be a vector bundle on $X$. We describe an explicit isomorphism between two different representations of the global $\Ext$ groups $\Ext^k(\mathcal{F},G)$. The first representation is given by the cohomology of a twisted complex in the sense of Toledo and Tong, and the second one is obtained from the Dolbeault complex associated with $G$. A key tool that we introduce for explicitly describing this isomorphism is a residue current associated with a twisted resolution of $\mathcal{F}$.
\end{abstract}
\maketitle
\section{Introduction}
Let $X$ be a complex manifold, and let $\cover$ be a Leray cover. If $G$ is a vector bundle over $X$, then the sheaf cohomology $H^k(X,G)$ can be represented as \v{C}ech cohomology $\check{H}^k(\cover,G)$ or as Dolbeault cohomology $H^{0,k}(X,G)$. If $(\rho_\alpha)$ is a partition of unity subordinate to $\cover$, then there is an explicit isomorphism between these representations, $\check{H}^k(\cover,G) \overset{\cong}{\to} H^{0,k}(X,G)$, given by
\begin{equation} \label{eq:cechDolbeault}
	[(c_{\alpha_0\dots \alpha_k})] \mapsto
	\left[ \sum_{(\alpha_0,\dots,\alpha_k)}
	\dbar\rho_{\alpha_0} \wedge \dots \wedge \dbar\rho_{\alpha_{k-1}}
	\rho_{\alpha_k} c_{\alpha_0 \dots \alpha_k} \right].
\end{equation}
One way to prove this is to use the partition of unity to construct an appropriate homotopy operator on the Dolbeault-\v{C}ech-complex as described in, for example, \cite{Harvey}*{Section~1}. This is completely analogous to the case of de Rham cohomology, which is described in, for example, \cite{BT}*{\S II.9}.

Recall that $H^k(X,G) \cong \Ext^k(\holo_X,G)$, so the above result could be viewed as an explicit isomorphism between different representations of $\Ext^k(\holo_X,G)$. In this paper, we will generalize this result to provide an explicit isomorphism between different representations of $\Ext^k(\mathcal{F},G)$, where $\mathcal{F}$ is any coherent $\holo_X$-module.

In \cite{TT1}, Toledo and Tong introduced the notion of a twisting cochain $(F,a)$, which gives a \emph{twisted complex} $(C^\bullet(\cover,\Homs^\bullet(F,G)),D^0,D^{\geq 1})$. A particular twisting cochain is a \emph{twisted resolution} of $\mathcal{F}$,
and the authors proved that the $k$th cohomology of the total complex with differential $D = D^0 + D^{\geq 1}$ of the twisted complex computes $\Ext^k(\mathcal{F},G)$, i.e.,
\[
	H^k\left(
	\bigoplus_{p+r=\bullet}
	C^p(\cover,\Homs^r(F,G))
	\right) \cong
	\Ext^k(\mathcal{F},G).
\]
If $\mathcal{F} = \holo_X$, then this twisted complex is just the ordinary \v{C}ech complex with respect to $G$.
If $\mathcal{F}$ has a locally free resolution $F$, then the cohomology of the total complex of the twisted complex is the hypercohomology
of $\Homs(F,G)$ as defined in
for example \cite{GH}*{p. 446}.

Analogously, there is also a generalization of Dolbeault cohomology to represent $\Ext^k(\mathcal{F},G)$ as the $k$th cohomology of $\Hom(\mathcal{F},\currs^{0,\bullet}(G))$.
The main result of this paper is to give an explicit isomorphism between these two representations of $\Ext^k(\mathcal{F},G)$. The key tool that we introduce for explicitly describing this isomorphism is a residue current $R$ associated with a twisted resolution $(F,a)$.
The residue current $R$ is a current with values in $C^\bullet(\cover,\Homs^\bullet(F,F))$, which we describe in Section~\ref{section:residue}.
This construction is a generalization of the residue current associated with a locally free resolution of a sheaf, as introduced by Andersson and Wulcan
in \cite{AW1}, which in turn is a generalization of the classical Coleff--Herrera product introduced in \cite{CH}.

\begin{theorem}
\label{thm:main}
Let $\mathcal{F}$ be a coherent $\holo_X$-module, and let $G$ be a holomorphic vector bundle on $X$. Let $(F,a)$ be a twisted resolution of $\mathcal{F}$,
and let $R$ be the associated residue current. Then there is a quasi-isomorphism
\begin{equation} \label{eq:qiso}
	\bigoplus_{p+r=\bullet} C^p(\cover,\Homs^r(F,G)) \to \Hom(\mathcal{F},\currs^{0,\bullet}(G))
\end{equation}
given by
\begin{equation} \label{eq:twisted-iso-formula}
	\xi \mapsto
	\sum_{j\geq 0} v_0 (\dbar v)^j (\xi R)^j.
\end{equation}
Here, $\xi R$ denotes the composition of $\xi$ and $R$ defined by \eqref{eq:pairing}, $(\xi R)^j$ denotes the component of \v{C}ech degree $j$,
cf., Section~\ref{ssect:notation},
$v$ and $v_0$ are the operators defined in Section~\ref{ssection:vdef}, and $(\dbar v)^j$ denotes the composition $\dbar\circ v$
repeated $j$ times, and $\Hom(\mathcal{F},\currs^{0,\bullet}(G))$ is assumed to be equipped with the differential $-\dbar$ (see Remark~\ref{rem:dbarSign}).
\end{theorem}

The fact that \eqref{eq:twisted-iso-formula} is a quasi-isomorphism means that it defines a morphism of
complexes which induces an isomorphism between the above mentioned representations of $\Ext^k(\mathcal{F},G)$:
\begin{equation}
\label{eq:twisted-iso}
	H^k
	\left(
	\bigoplus_{p+r=\bullet}
	C^p(\cover,\Homs^r(F,G))
	\right) \overset{\cong}{\to}
	H^k(\Hom(\mathcal{F},\currs^{0,\bullet}(G))).
\end{equation}

In case $\mathcal{F}$ is locally free, and one identifies $\mathcal{F}$ as a twisting cochain concentrated in degree $0$, then
the associated residue current $R$ equals the identity, and the isomorphism \eqref{eq:twisted-iso-formula} becomes
$\xi \mapsto v_0 (\dbar v)^k \xi$, which reduces to the isomorphism \eqref{eq:cechDolbeault}.

Theorem~\ref{thm:main} may be considered as a global version of a corresponding statement for sheaf Ext, $\Exts^k(\mathcal{F},\omega)$, where an explicit isomorphism between different representations is provided by residue currents, see Proposition~\ref{prop:Exts}. This result is due to Andersson, \cite{And1}, generalizing earlier work by Dickenstein--Sessa, \cite{DS}, in the case of $\mathcal{F}=\holo_Z$ being the structure sheaf of a complete intersection $Z$.

One situation in which the different realizations of Ext groups in Theorem~\ref{thm:main} appear is in connection with Serre duality, as originally proven for locally free sheaves, in \cite{Serre}, and extended to coherent $\holo_X$-modules by work of Malgrange, \cite{MalSerre}. This version of Serre duality concerns a perfect pairing between Ext groups and sheaf cohomology, and for the two different realizations of Ext in Theorem~\ref{thm:main}, and corresponding realizations for sheaf cohomology, one may define two different such pairings. It turns out that the isomorphism \eqref{eq:twisted-iso} is compatible with these two realizations
of the pairing, as is elaborated on in Section~\ref{sect:serre}.

Twisted resolutions generalize locally free resolutions, and loosely speaking consist of local free resolutions glued
together in an appropriate way. In contrast to locally free resolutions one may always find globally a twisted resolution of
any coherent $\holo_X$-module $\mathcal{F}$.
The residue current we introduce associated with a twisted resolution consists of one part which
is the residue currents as defined in \cite{AW1} associated with the local free resolutions that are glued together, and
additional parts so that the whole residue current have various desirable properties, cf.,
Proposition-Definition~\ref{prop:URdef}, Theorem~\ref{thm:R} and Remark~\ref{rem:Rproperties}.

This paper is organized as follows. We begin, in Section~\ref{section:twisting-cochains}, by giving a brief introduction to the notion of twisting cochains. Here we also make some small adaptations so that we can define currents associated with these objects. Residue currents are best described using the language of almost semi-meromorphic and pseudomeromorphic currents. In Section~\ref{section:pseudo}, we recall the necessary definitions and results about these currents that we shall need in this paper. In Section~\ref{section:residue} we define a residue current associated with a twisting cochain, and we prove that it satisfies certain properties, which we then use to prove our main theorem in Section~\ref{section:main-theorem}. Finally, as mentioned above, in Section~\ref{sect:serre}, we discuss how Theorem~\ref{thm:main} fits into the context of Serre duality.

\section{Twisting cochains}
\label{section:twisting-cochains}
Throughout this paper, $X$ will denote a complex manifold of dimension $n$, and $\cover = (\opens_\alpha)$ will denote a covering of $X$ by Stein open sets.
We will use the notation $\opens_{\alpha_0 \dots \alpha_p} \defeq \opens_{\alpha_0} \cap \dots \cap \opens_{\alpha_p}$.

In this section, we will mainly recall the relevant parts about twisting cochains, twisted resolutions and twisted complexes from \cite{TT1}.
This material is essentially the same as described in Sections 1 and 2 of \cite{TT1}. 
In order to incorporate this theory with the theory of residue currents from \cite{AW1}, and get our sign convention consistent
with both \cite{TT1} and \cite{AW1}, we consider current-valued analogues of the objects described in \cite{TT1}.

Let $F \defeq (F_\alpha)$ and $G \defeq (G_\alpha)$ be families of bounded graded holomorphic vector bundles over $\opens_\alpha$.
Recall that
\[
	\Homs^r(F_\alpha,G_\beta) =
	\bigoplus_j
	\Homs(F_\alpha^{-j},G_\beta^{-j+r}).
\]
Let $\currs^{0,q}(\Homs^r(F_\alpha,G_\beta))$ denote the sheaf over $\opens_{\alpha \beta}$ of $\Homs^r(F_\alpha,G_\beta)$-valued $(0,q)$-currents.
We define a product
\[
	\currs^{0,q}(\Homs^r(F_\beta,G_\gamma)) \times
	\currs^{0,q'}(\Homs^{r'}(E_\alpha,F_\beta)) \to
	\currs^{0,q+q'}(\Homs^{r+r'}(E_\alpha,G_\gamma))
\]
over $\opens_{\alpha \beta \gamma}$.
For decomposable sections $\eta \otimes f$ and $\tau \otimes g$, i.e.,
where $\eta$ and $\tau$ are $(0,q)$ and $(0,q')$-currents, respectively,
and $f$ and $g$ are sections of $\Homs^r(F_\beta,G_\gamma)$ and
$\Homs^{r'}(E_\alpha,F_\beta)$, respectively,
the product is defined as
\begin{equation}
\label{eq:prod1}
	(\eta \otimes f)(\tau \otimes g) \defeq
	(-1)^{rq'}
	(\eta \wedge \tau) \otimes fg,
\end{equation}
provided that either $\eta$ or $\tau$ is smooth.
This product is defined in a way consistent with the sign convention of the super structure
in \cite{AW1}*{Section 2}.
We will consider a sort of \v{C}ech cochains with coefficients in these sheaves. We define
\begin{equation}
\label{eq:cechCurrentHom}
	C^p(\cover,\currs^{0,q}(\Homs^r(F,G))) \defeq
	\prod_{(\alpha_0, \dots, \alpha_p)}
	\currs^{0,q}(\Homs^r(F_{\alpha_p},G_{\alpha_0}))
	({\opens_{\alpha_0 \dots \alpha_p}}).
\end{equation}
For an element $f \in C^p(\cover,\currs^{0,q}(\Homs^r(F,G)))$, we define its \emph{(total) degree} as $\deg f = p+q+r$,
and we call $p$ the \emph{\v{C}ech degree}, $q$ the \emph{current degree} and $r$ the \emph{Hom degree}.

In this paper the families of vector bundles that we shall consider will be concentrated in nonpositive degree. We use this convention to be consistent with the conventions in \cite{TT1}. In \cite{AW1}, complexes of vector bundles are considered to be concentrated in nonnegative degree, and the differential is assumed to be decreasing the degree, in contrast to \cite{TT1}.

We define a bilinear product
\begin{multline}
    \label{eq:pairing}
	   C^p(\cover,\currs^{0,q}(\Homs^r(F,G))) \times
	   C^{p'}(\cover,\currs^{0,q'}(\Homs^{r'}(E,F))) \\ \to
	   C^{p+p'}(\cover,\currs^{0,q+q'}(\Homs^{r+r'}(E,G))),
\end{multline}
which maps $(f,g)$ to the product $fg$ defined by
\begin{equation}
\label{eq:prod2}
	(fg)_{\alpha_0 \dots \alpha_{p+p'}} \defeq
	(-1)^{(q+r)p'}
	f_{\alpha_0 \dots \alpha_p} g_{\alpha_p \dots \alpha_{p+p'}},
\end{equation}
where the product on the right-hand side is defined by \eqref{eq:prod1}.
We will consider three differentials acting on $C^\bullet(\cover,\currs^{0,\bullet}(\Homs^\bullet(F,G)))$,
$D$, $\delta$ and $\dbar$, and the signs in \eqref{eq:prod1} and \eqref{eq:prod2} are chosen in a way that
makes all these operators into antiderivations with respect to the product \eqref{eq:prod2}.
Informally, one could consider the elements to have the ``\v{C}ech part'' to the left, the ``current part'' in the middle, and the ``$\Homs$ part'' to the right,
and these different parts anticommute according to the degree in the products \eqref{eq:prod1} and \eqref{eq:prod2}.

We will now describe how the $\dbar$-operator and an analogue of the \v{C}ech coboundary acts on $C^\bullet(\cover,\currs^{0,\bullet}(\Homs^\bullet(F,G)))$.
We let the $\dbar$-operator act as an operator of degree 1 on $C^\bullet(\cover,\currs^{0,\bullet}(\Homs^\bullet(F,G)))$ by
\[
	(\dbar f )_{\alpha_0 \dots \alpha_p} \defeq
	(-1)^p \dbar f_{\alpha_0 \dots \alpha_p}.
\]
With this definition we have that $\dbar(fg) = (\dbar f)g + (-1)^{\deg f} f (\dbar g$), and as usual $\dbar^2 = 0$.

Next we define an operator of degree 1,
\[
	\delta: C^p(\cover,\currs^{0,q}(\Homs^r(F,G))) \to
	C^{p+1}(\cover,\currs^{0,q}(\Homs^r(F,G))),
\]
by
\[
	(\delta f)_{\alpha_0 \dots \alpha_{p+1}} \defeq
	\sum_{k=1}^p (-1)^k
	f_{\alpha_0 \dots \widehat{\alpha}_k \dots \alpha_{p+1}}
	|_{\opens_{\alpha_0 \dots \alpha_{p+1}}}.
\]
In particular, $\delta f = 0$ for any $f \in C^0(\cover,\currs^{0,q}(\Homs^r(F,G)))$.
Note that $\delta$ is similar to the usual \v{C}ech coboundary, but in the sum, it is necessary to omit $f_{\alpha_1 \dots \alpha_{p+1}}$ and $f_{\alpha_0 \dots \alpha_p}$ since these do not belong to $\Homs^r(F_{\alpha_{p+1}},F_{\alpha_0})$. However, we still have that $\delta$ is a differential and an antiderivation with respect to the product \eqref{eq:prod2}, i.e., $\delta^2 = 0$, and
\[
	\delta(fg) = (\delta f) g + (-1)^{\deg f} f (\delta g).
\]

We are now ready to define the notion of a twisting cochain. We define
\[
	C^p(\cover,\Homs^r(F,G)) \defeq
	\prod_{(\alpha_0, \dots, \alpha_p)}
	\Homs^r(F_{\alpha_p},G_{\alpha_0})
	({\opens_{\alpha_0 \dots \alpha_p}}).
\]
Note that this is the subgroup of $\dbar$-closed elements of $C^p(\cover,\currs^{0,0}(\Homs^r(F,G)))$.
This subgroup is the group of twisted \v{C}ech cochains considered in \cite{TT1}, and when we restrict to this subgroup,
the sign conventions here coincide with the ones in \cite{TT1}.

\begin{definition}
\label{def:twisting-cochain}
A \emph{twisting cochain} $a \in C^\bullet(\cover,\Homs^\bullet(F,F))$ is an element $a = \sum_{k \geq 0} a^k$, where $a^k \in  C^k(\cover,\Homs^{1-k}(F,F))$,
such that
\begin{equation}
\label{eq:twisting-cochain}
	\delta a + aa = 0,
\end{equation}
and $a_{\alpha \alpha}^1 = \id_{F_\alpha}$ for all $\alpha$. For simplicity we shall simply refer to the pair $(F,a)$ as a twisting cochain.
\end{definition}

Recall that a holomorphic vector bundle is defined by a 1-cocycle. A twisting cochain is a generalization of this. By \eqref{eq:twisting-cochain}, $a$ must satisfy
\begin{align}
	&a_\alpha^0 a_\alpha^0 = 0 \label{eq:twisted0} \\
	&a_\alpha^0 a_{\alpha \beta}^1 = a_{\alpha \beta}^1 a_\beta^0 \label{eq:twisted1} \\
	&a_{\alpha \gamma}^1 - a_{\alpha \beta}^1 a_{\beta \gamma}^1 =
	a_\alpha^0 a_{\alpha \beta \gamma}^2 + a_{\alpha \beta \gamma}^2 a_\gamma^0 \label{eq:twisted2}.
\end{align}
The first equation says that $(F_\alpha,a^0)$ is a complex, the second says that $a_{\alpha \beta}^1$ defines a chain map $(F_\beta|_{\opens_{\alpha \beta}},a^0) \to (F_\alpha|_{\opens_{\alpha \beta}},a^0)$, and the third says that, over $\opens_{\alpha \beta \gamma}$, $a_{\alpha \gamma}^1$ and $a_{\alpha \beta}^1 a_{\beta \gamma}^1$ are chain homotopic, with the homotopy given by $a^2_{\alpha\beta\gamma}$. In particular, from the condition $a_{\alpha \alpha}^1 = \id_{F_\alpha}$, it follows that $a_{\alpha \beta}^1$ and $\alpha_{\beta \alpha}^1$ are chain homotopy inverses to each other.

Thus for each $\alpha$ we have cohomology sheaves $\mathcal{H}^\bullet_{a_\alpha^0}(F_\alpha)$ over $\opens_\alpha$, and over each intersection $\opens_{\alpha \beta}$ we have an isomorphism $H(a_{\alpha \beta}^1): \mathcal{H}^\bullet_{a_\beta^0}(F_\beta)|_{\opens_{\alpha \beta}} \to \mathcal{H}^\bullet_{a_\alpha^0}(F_\alpha)|_{\opens_{\alpha \beta}}$ such that over each $\opens_{\alpha \beta \gamma}$, $H(a_{\alpha \beta}^1) H(a_{\beta \gamma}^1) = H(a_{\alpha \gamma}^1)$. We denote by $\mathcal{H}_a^\bullet$ the sheaf that we obtain by gluing the sheaves $\mathcal{H}^\bullet_{a_\alpha^0}$ via these isomorphisms.

The twisting cochains that are of interest in this paper arise in the following way.

\begin{definition}[A twisted resolution of $\mathcal{F}$]
Let $\mathcal{F}$ be a coherent $\holo_X$-module.
By the syzygy theorem one can find a cover $\cover$ such that for each $\alpha$ there exists a free resolution of $\mathcal{F}|_{\opens_\alpha}$,
\[
	\dots
	\overset{a_\alpha^0}{\longrightarrow} F_\alpha^{-1}
	\overset{a_\alpha^0}{\longrightarrow} F_\alpha^0
	\longrightarrow \mathcal{F}|_{\opens_\alpha}
	\longrightarrow 0,
\]
of length at most $\dim X$. Over each intersection $\opens_{\alpha \beta}$ one can find a chain map $a_{\alpha \beta}^1: (F_\beta|_{\opens_{\alpha \beta}},a_\beta^0) \to (F_\alpha|_{\opens_{\alpha \beta}},a_\alpha^0)$ that extend the identity morphism on $\mathcal{F}|_{\opens_{\alpha \beta}}$, and which can be chosen to be the identity if $\alpha = \beta$. Since $a_{\alpha \gamma}^1$ and $a_{\alpha \beta}^1 a_{\beta \gamma}^1$ are chain maps $(F_\gamma|_{\opens_{\alpha \beta \gamma}},a_\gamma^0) \to (F_\alpha|_{\opens_{\alpha \beta \gamma}},a_\alpha^0)$ that extends the identity morphism on $\mathcal{F}|_{\opens_{\alpha \beta \gamma}}$, there exists a chain homotopy $a_{\alpha \beta \gamma}^2$ between these maps. As explained in \cite{OTT}*{Section~1.3}, one can proceed inductively to construct a twisting cochain $a = \sum_k a^k$. Note that $\mathcal{H}_a^k = 0$ if $k \neq 0$ and $\mathcal{H}_a^0 \cong \mathcal{F}$. We shall refer to $(F,a)$ as a \emph{twisted resolution of $\mathcal{F}$}.
\end{definition}

\begin{example} \label{ex:globalRes}
    Assume that $\mathcal{F}$ actually has a global locally free resolution $(E,\varphi)$ of finite length.
    Then we may in the construction above always choose restrictions of the given resolution as the local resolutions,
    i.e., the twisted resolution is defined by $F_\alpha \defeq E|_{\opens_\alpha}$, $a = a^0+a^1$, where
    $a^0_\alpha = \varphi|_{\opens_\alpha}$ and $a^1_{\alpha\beta} = \id_{E|_{\opens_{\alpha\beta}}}$.
    Conversely, it follows by \eqref{eq:twisted2} that for any twisted resolution $(F,a)$ of $\mathcal{F}$ with $a^2=0$,
    $a^1$ defines a $1$-cocycle for each $F^k_\bullet$, and may thus be used to define graded holomorphic
    vector bundles, and by \eqref{eq:twisted0} and \eqref{eq:twisted1}, these graded vector bundles may be turned into a complex
    with the differential defined by $a^0$.
\end{example}

Consider two twisting cochains $(F,a)$ and $(G,b)$. We define an operator $D$ of degree 1 on $C^\bullet(\cover,\currs^{0,\bullet}(\Homs^\bullet(F,G)))$,
\[
	Df \defeq \delta f + bf - (-1)^{\deg f} fa.
\]
We have that $D^2=0$, and
\begin{equation} \label{eq:Dderivation}
	D(fg) = (Df)g + (-1)^{\deg f} f (Dg).
\end{equation}

Note that if $(F,a)$ and $(G,b)$ are global complexes of vector bundles as Example~\ref{ex:globalRes},
and $a^1=\id=b^1$, then the part $D^1$ of $D$ of \v{C}ech degree $1$, $D^1 f = \delta f + b^1 f - (-1)^{\deg f} f a^1$
equals the usual \v{C}ech differential, i.e., $a^1$ and $b^1$ ``compensate'' for indices $k=0$ and $k=p+1$ that
are not included in the definition of $\delta$ above. In particular, in this situation,
$H^k(\bigoplus_{p+r=\bullet} C^p(\cover,\Homs^r(F,G)))$ equals the usual hypercohomology,
as in for example \cite{GH}*{p. 446}.

We have that $\dbar D = -D \dbar$. We combine $D$ and $\dbar$ into an operator $\nabla = D - \dbar$ of degree 1 on $C^\bullet(\cover,\currs^{0,\bullet}(\Homs^\bullet(F,G)))$. Clearly, we have that $\nabla^2 = 0$, and
\begin{equation} \label{eq:nablaDerivation}
	\nabla(fg) = (\nabla f) g + (-1)^{\deg f} f (\nabla g).
\end{equation}

If $G$ is a vector bundle, then we can view it as a family of graded vector bundles concentrated in degree 0, $G = (G_\alpha)$, where $G_\alpha = G|_{\opens_\alpha}$, and we have that a twisting cochain for $G$ is given by $b = b^1 = \id$.
We will tacitly do this kind of identification when for example $(F,a)$ is a twisting cochain and $G$ is a vector bundle,
and we consider groups like $C^p(\cover,\currs^{0,q}(\Homs^r(F,G)))$ or $C^p(\cover,\Homs^r(F,G))$, as for example
in Theorem~\ref{thm:main}.

\subsection{Notation for parts of an element $f \in C^p(\cover,\currs^{0,q}(\Homs^r(F,G)))$}
\label{ssect:notation}

We denote by $f_k^\ell$ the parts of $f$ that belong to
\[
\prod_{(\alpha_0, \dots, \alpha_p)} \currs^{0,q}(\Homs(F_{\alpha_p}^{-\ell},G_{\alpha_0}^{-k}))
(\opens_{\alpha_0 \dots \alpha_p}),
\]
i.e., when using both a superscript and a subscript, we pick out morphisms between bundles
in certain degrees. We will also use the notation \[f^\ell_\bullet \defeq \sum_k f^\ell_k.\]
We say that \emph{$f$ takes values in $\Hom(F^{-\ell},F^{-k})$} if $f=f^\ell_k$,
and that \emph{$f$ takes values in $\Hom(F^{-\ell},F)$} if $f=f^\ell_\bullet$.

We denote by $f^p$ the parts of $f$ that belong to
\[
\prod_{(\alpha_0, \dots, \alpha_p)} \currs^{0,\bullet}(\Homs(F_{\alpha_p},G_{\alpha_0}))
({\opens_{\alpha_0 \dots \alpha_p}}),
\]
i.e., when using a single superscript, we pick out a certain \v{C}ech degree.

\section{Pseudomeromorphic and almost semi-meromorphic currents}
\label{section:pseudo}
We will use the language of residues of almost semi-meromorphic currents
as introduced in \cite{AW3} to describe the currents we study in this paper.

Let $s$ be a holomorphic section of a Hermitian holomorphic line bundle $L$ on $X$,
The \emph{principal value current} $[1/s]$ can be defined as $[1/s] \defeq \lim_{\epsilon \to 0} \chi(|s|^2/\epsilon)\frac{1}{s}$,
where $\chi : \R_{\geq 0} \to \R_{\geq 0}$ is a smooth cut-off function, i.e., $\chi(t) \equiv 0$ for $t \ll 1$ and $\chi(t) \equiv 1$ for $t \gg 1$.
A current is \emph{semi-meromorphic} if it is of the form $[\omega/s] \defeq \omega[1/s]$, where $\omega$ is a smooth form with values in $L$.
A current $a$ is \emph{almost semi-meromorphic} on $X$, written $a \in ASM(X)$, if there is a modification $\pi : X' \to X$ such that
$$a=\pi_*(\omega/s),$$ where $\omega/s$ is a semi-meromorphic current in $X'$.
Recall that a modification is a proper surjective holomorphic map $\pi : X' \to X$, where $X'$ and $X$ are complex spaces, such that
there exists a nowhere dense analytic subset $E \subseteq X$ such that $\pi|_{X'\setminus \pi^{-1}(E)} : X' \setminus \pi^{-1}(E) \to X \setminus E$ is a biholomorphism.
The almost semi-meromorphic currents on $X$ form an algebra over
smooth forms.

Almost semi-meromorphic functions are special cases of so-called pseudomeromorphic currents, as introduced in \cite{AW2}.
The class of pseudomeromorphic currents is closed under multiplication with smooth forms and under $\dbar$.
One important property of pseudomeromorphic currents is that they satisfy the following \emph{dimension principle}.

\begin{proposition} \label{prop:dimPrinciple}
Let $T$ be a pseudomeromorphic $(*,q)$-current on $X$ with support on an analytic subset $Z$.
If $\codim Z > q$, then $T = 0$.
\end{proposition}

Given a pseudomeromorphic current $T$ and an analytic subset $V$, in \cite{AW2} was introduced a restriction of $T$ to $X\setminus V$,
which is a pseudomeromorphic current on $X$ defined by $\mathbf{1}_{X \setminus V} T \defeq \lim_{\epsilon \to 0} \chi(|F|^2/\epsilon) T$, where $\chi$
is a cut-off function as above, and $F$ is a section of a holomorphic vector bundle such that $V = \{ F = 0 \}$.
A pseudomeromorphic current $T$ on $X$ is said to have the standard extension property (SEP) if $\mathbf{1}_{X \setminus V} T = T$
for any analytic subset $V$ of positive codimension.
It follows from the dimension principle and the fact that the restriction commutes with multiplication with smooth forms that
almost semi-meromorphic currents have the SEP. In particular, if $\alpha$ is a smooth form on $X \setminus V$, and $\alpha$ has
an extension as an almost semi-meromorphic current $a$ on $X$, then the extension is given by
\begin{equation} \label{eq:asmExtension}
    a = \lim_{\epsilon \to 0} \chi(|F|^2/\epsilon) \alpha.
\end{equation}

Let $a$ be an almost semi-meromorphic current on $X$. Let $Z$ be the smallest analytic subset of $X$ of positive codimension such that $a$
is smooth outside of $Z$. By \cite{AW3}*{Proposition~4.16}, $\mathbf{1}_{X \setminus Z} \dbar a$ is almost semi-meromorphic on $X$.
The \emph{residue} $R(a)$ of $a$ is defined by
\begin{equation*}
	R(a) \defeq \dbar a - \mathbf{1}_{X \setminus Z}\dbar a.
\end{equation*}

The proposition \cite{AW3}*{Proposition~4.16} mentioned above may be slightly reformulated in the following way, which will be useful later.
\begin{proposition} \label{prop:dbarAsmExtension}
	Let $Z$ be an analytic subset of $X$ of positive codimension, and let $\alpha$ be a smooth form on $X \setminus Z$ which
	admits an extension as an almost semi-meromorphic current to $X$. Then $\dbar \alpha$ admits an extension as an almost semi-meromorphic
	current to $X$.
\end{proposition}

Note that
\begin{equation} \label{eq:resSupport}
    \supp R(a) \subseteq Z.
\end{equation}
Since $a$ is almost semi-meromorphic, and thus has the SEP, it follows
by \eqref{eq:asmExtension} that
\begin{equation}
\label{eq:residue}
    R(a)=\lim_{\epsilon \to 0}
    \left(\dbar(\chi_\epsilon a) - \chi_\epsilon \dbar a \right)
    = \lim_{\epsilon \to 0} \dbar\chi_\epsilon \wedge a,
\end{equation}
where $\chi$ is as above, $F$ is a section of a vector bundle such that $\{ F = 0 \} \supseteq Z$, and $F \not\equiv 0$
and $\chi_\epsilon = \chi(|F|^2/\epsilon)$.
It follows directly from for example \eqref{eq:residue} that if $\psi$ is a smooth form, then
\begin{equation} \label{eq:residueSmooth}
    R(\psi \wedge a) = (-1)^{\deg \psi} \psi\wedge R(a).
\end{equation}

For elements in (\ref{eq:cechCurrentHom}), we will say that they are almost semi-meromorphic and pseudomeromorphic respectively if their components are. Moreover the residue of an element in (\ref{eq:cechCurrentHom}) is defined as the residues of the components.

\begin{example} \label{eq:URgenExact}
Let $(E,\varphi)$ be a generically exact complex of locally free $\holo_X$-modules
\begin{equation} \label{eq:genExactComplex}
	0 \to E^{-N} \overset{\varphi_{-N}}{\longrightarrow} \dots \overset{\varphi_{-1}}{\longrightarrow} E^0 \to 0,
\end{equation}
and assume that each $E^k$ is equipped with some hermitian metric. Let $\sigma_{-k}$ denote the minimal right-inverse of $\varphi_{-k}$, cf., the next section,
and let $\sigma := \sum_{k=1}^N \sigma_{-k}$. The minimal right-inverses are smooth outside of the set $Z_k$ where the rank of $\varphi_{-k}$ drops from its generic rank,
so $\sigma_{-k}$ is smooth outside of $Z_k$. Following \cite{AW1}*{Section~2}, one may find a resolution $\pi : \tilde{X} \to X$ of $X$ such that locally on $\tilde{X}$,
there is a holomorphic function $f_{-k,0}$ defining the ideal of the $r_k\times r_k$-minors
of $\varphi_{-k}$, where $r_k$ is the generic rank of $\varphi_{-k}$ and a smooth
$\End(E)$-valued section $\sigma_{-k}'$ such that $\pi^* \sigma_{-k} = \sigma_{-k}'/f_{-k,0}$.
Thus, $\sigma$ is an almost
semi-meromorphic section of $\End(E)$. More generally, the form $u = \sum_{\ell=1}^N \sigma(\dbar\sigma)^{\ell-1}$, a priori defined as a smooth form outside of $Z = \cup Z_k$,
admits a similar decomposition, cf., \cite{AW1}*{(2.8)}, and may thus be extended to an almost semi-meromorphic current $U^E$ on all of $X$.
The residue current $R^E$ associated to $(E,\varphi)$, as defined in \cite{AW1}, is $R^E = \id_E - \nabla U^E$, where $\nabla = \varphi-\dbar$.
Alternatively, $R^E$ is in this situation given by $R^E=R(U^E)$, cf., i.e., the discussion in \cite{Lar}*{Section~2.4}.

In case \eqref{eq:genExactComplex} is exact in degree $<0$, i.e., such that it provides a resolution of $\coker \varphi_{-1}$, then $R^E$ satisfies the fundamental
duality principle that a section $\xi$ of $E_0$ belongs to $\im \varphi_1$ if and only if $R^E \xi = 0$.
\end{example}

One of the basic examples of such currents is the following.

\begin{example}
	Let $f=(f_1,\dots,f_p)$ be a tuple of holomorphic functions defining a complete intersection, i.e., such that $\codim \{ f_1 = \dots = f_p = 0 \} = p$.
	Then one may define the so-called Coleff-Herrera product of $f$, as introduced (in a slightly different way) in \cite{CH} by
	\begin{equation*}
		\dbar \frac{1}{f_p}\wedge \dots \wedge \dbar \frac{1}{f_1} := \lim_{\epsilon_p \to 0} \dots \lim_{\epsilon_1 \to 0} \frac{\dbar\chi(|f_p|^2/\epsilon_p)}{f_p} \wedge \dots \wedge \frac{\dbar\chi(|f_1|^2/\epsilon_1)}{f_1},
	\end{equation*}
	where $\chi$ is a smooth cut-off functions as above.

	Given $f$, there is an associated complex of vector bundles, the Koszul complex $(K,\psi) = (\bigwedge \holo^{\oplus p}, \delta_f)$, which provides a resolution of $\holo/\mathcal{J}(f)$ when $f$ defines a complete intersection.
	If one equips $\holo^{\oplus p}$ with the trivial metric, then the current $U^K$ associated to $K$ equals the Bochner-Martinelli kernel of $f$, and the associated residue current $R^K$ equals the Bochner-Martinelli
	residue current as introduced in \cite{PTY}, which in turn equals the Coleff-Herrera product of $f$.
\end{example}

More generally, coherent $\holo_X$-modules which are Cohen-Macaulay yield residue currents which may at least locally be described in terms of Coleff-Herrera products.

\begin{example}
	Let $\mathcal{G}$ be a coherent $\holo_X$-module, and assume that $\mathcal{G}$ is Cohen-Macaulay. Thus, $\mathcal{G}$ admits locally
	a free resolution \eqref{eq:genExactComplex}, where the length $N$ equals $p=\codim (\supp \mathcal{G})$.
	Let $(E,\varphi)$ be such a resolution. We may furthermore locally find a complete intersection $(f_1,\dots,f_p)$
	and a surjective morphism $(\holo/\mathcal{J}(f_1,\dots,f_p))^{\oplus r} \to \mathcal{G}$, where $r = \rank E_0$.
	If one lets $(K,\psi)$ be the direct sum of $r$ copies of the Koszul complex of $f$, then one may find a morphism
	of complexes $a : (K,\psi) \to (E,\varphi)$ which extends the morphism $(\holo/\mathcal{J}(f_1,\dots,f_p))^{\oplus r} \to \mathcal{G}$.
	Then, the comparison formula for residue currents, see \cite{Lar}*{Section~4}, yields that $R^E$ is given
	by\footnote{Up to an identification of $E_0$ with $K_0$, which are free $\holo$-modules of the same rank}
	$a_p R^K$, i.e., it is given by the holomorphic matrix $a_p$ times the Coleff-Herrera product of $f$.
\end{example}

\section{A residue current associated with a twisted resolution}
\label{section:residue}
In \cite{AW1}, Andersson and Wulcan constructed a residue current from a locally free resolution of a sheaf of positive codimension. In this section we will generalize this construction to construct a residue current from a twisted resolution of a coherent $\holo_X$-module.

Throughout this section, let $(F,a)$ be a twisted resolution of a coherent $\holo_X$-module $\mathcal{F}$.
We will tacitly assume that the bundles $F_\alpha^k$ are equipped with Hermitian metrics.

Let $Z$ be the smallest analytic subset of $X$ such that $\mathcal{F}|_{X \setminus Z}$ is a vector bundle.
For each $\alpha$, let $\opens_\alpha' \defeq \opens_\alpha \setminus Z$, and define the cover $\cover' \defeq (\opens_\alpha')$ of $X \setminus Z$. We define an element $\sigma^0 \in C^0(\cover',\currs^{0,0}(\Homs^{-1}(F,F)))$ in the following way. For each $\alpha$, let $\sigma_\alpha^0$ be the minimal right-inverse of $a_\alpha^0$ on $\opens_\alpha'$, i.e., the Moore--Penrose inverse, which can be defined by the properties $a_\alpha^0 \sigma_\alpha^0 a_\alpha^0 = a_\alpha^0$, $\sigma_\alpha^0|_{(\im a_\alpha^0)^\perp}= 0$, and $\im \sigma_\alpha^0 \perp \ker a_\alpha^0$. From the last two properties it follows that $(\sigma^0)^2 = 0$.
We write $a = a^0 + a^{\geq 1}$, i.e., $a^{\geq 1} = \sum_{k\geq 1} a^k$.
Define
\begin{equation} \label{eq:sigmadef}
	\sigma \defeq \sigma^0(\id+a^{\geq 1} \sigma^0)^{-1} = \sigma^0-\sigma^0 a^{\geq 1}\sigma^0 + \sigma^0 a^{\geq 1}\sigma^0a^{\geq 1}\sigma^0 - \dots,
\end{equation}
where the sum in the right-hand side is finite since $a^{\geq 1}\sigma^0$ has negative Hom degree, and thus is nilpotent since the complexes have finite length.
Throughout this section, $\id$ denotes the identity element on $F$ in $C^0(\cover,\Homs^0(F,F))$ or $C^0(\cover',\Homs^0(F,F))$.
Since $(\sigma^0)^2 = 0$, it follows that $\sigma^2 = 0$.
\begin{propdef}
\label{prop:URdef}
Define
\begin{equation} \label{eq:udef}
    u \defeq \sigma(\id-\dbar\sigma)^{-1} = \sigma+\sigma\dbar\sigma+\sigma(\dbar\sigma)^2+ \dots.
\end{equation}
Then
\begin{equation} \label{eq:Udef}
    U \defeq \lim_{\epsilon \to 0} \chi_\epsilon u
\end{equation}
is an almost semi-meromorphic extension of $u$. Moreover,
\begin{equation} \label{eq:Rdef}
	R \defeq \id - \nabla U
\end{equation}
is pseudomeromorphic and $\nabla$-closed, i.e., $\nabla R = 0$.
\end{propdef}
We shall refer to $R$ as the residue current associated with $(F,a)$.
Note that $\sigma^0$ has degree $-1$, and $a^{\geq 1}$ has degree $1$, so
it follows by \eqref{eq:sigmadef}, \eqref{eq:udef}, and \eqref{eq:Udef}
that $\sigma$, $u$, and $U$ have degree $-1$.
Since $\nabla$ is an operator of degree $1$, it follows that $R$ has degree $0$.

\begin{proof}
    On $\opens_\alpha'$, $\sigma^0_\alpha$ coincides with the form $\sigma$ associated with $F_\alpha^\bullet$
    defined in \cite{AW1}. As explained in \cite{AW3}*{Example 4.18}, this form has
    an extension as an almost semi-meromorphic current on $\opens_\alpha$. By \eqref{eq:sigmadef}, $\sigma$
    is a sum of products of $\sigma^0$ and $a^{\geq 1}$, and since the almost semi-meromorphic currents
    form an algebra over smooth forms and may be restricted to open subsets, and $a^{\geq 1}$ is holomorphic,
    it follows that $\sigma$ has an almost semi-meromorphic extension.
    In addition, by \cite{AW3}*{Proposition~4.16} (or the argument in \cite{AW3}*{Example 4.18}), $\dbar\sigma$ has an
    almost semi-meromorphic extension, so by \eqref{eq:udef}, it follows that $u$ has an
    almost semi-meromorphic extension. By \eqref{eq:asmExtension}, this extension $U$ is given by
    \eqref{eq:Udef}.

    We have that $U$ is almost semi-meromorphic, and thus pseudomeromorphic.
Thus $\nabla U$ is also pseudomeromorphic since the class of pseudomeromorphic currents is
preserved by $\dbar$, multiplication with smooth forms, and restrictions to open subsets.
The same then holds for $R$ by \eqref{eq:Rdef}.
Since $\nabla \id = 0$ and $\nabla^2=0$, it follows that $\nabla R = 0$.
\end{proof}
The main aim of this section is to prove that the residue current has the following properties.
\begin{theorem}
\label{thm:R}
Let $(F,a)$ be a twisted resolution of a coherent $\holo_X$-module $\mathcal{F}$, and let $R$ be its associated residue current. Then
\begin{enumerate}[\normalfont(a)]
	\item \label{item:HomF0}
$R$ takes values in $\Hom(F^0,F)$,
\item
	$R_k^0 = 0$ for $k < \codim \mathcal{F}$ (where $\codim \mathcal{F}$ denotes $\codim (\supp \mathcal{F})$),
\item
$R a^0 = 0$, and
\item
$\supp R \subseteq \supp \mathcal{F}$.
\end{enumerate}
\end{theorem}

\begin{remark}
For simplicity, we will only consider residue currents associated with a coherent $\holo_X$-module,
since this assumption is crucial in Theorem~\ref{thm:R}, while with minor adaptions, one could in
Proposition-Definition~\ref{prop:URdef} consider residue currents associated with arbitrary twisting cochains,
not necessarily twisted resolutions of coherent $\holo_X$-modules.
\end{remark}

\begin{remark} \label{rem:Rproperties}
In the proof of Theorem~\ref{thm:main}, we rely on two crucial properties of $R$, first of all that $R$ is a pseudomeromorphic current
which is $\nabla$-cohomologous to the identity morphism on $(F,a)$, and secondly that $R$ takes values in $\Hom(F^0,F)$, i.e., part (\ref{item:HomF0}) of Theorem~\ref{thm:R}.
The first property is trivially guaranteed by \eqref{eq:Rdef} for any choice of $U$.
The point of the choice of $\sigma$, $u$ and $U$ in \eqref{eq:sigmadef}, \eqref{eq:udef} and \eqref{eq:Udef},
is that with these choices, also the second property is fulfilled.

Our definition is very much inspired by the construction of residue currents from \cite{AW1} which satisfy
similar properties in cases when there exists a global locally free resolution.
One crucial property of the residue currents defined in \cite{AW1} is that they satisfy a duality principle,
\cite{AW1}*{Theorem~1.1}. Our current $R$ also satisfy a similar duality principle, which is essentially a
consequence of the above two crucial properties that $R$ satisfies, see \cite{J}*{Theorem 1.1}.
\end{remark}

For the proof of Theorem~\ref{thm:R} we will need the following result, which provides a useful decomposition of the residue current.

\begin{proposition} \label{prop:Rasm}
    Let $R$ be the current defined by \eqref{eq:Rdef}. Then, $R$ can be decomposed as
    \begin{equation*}
        R = R(U) + R',
    \end{equation*}
    where $R(U)$ is the residue in the sense of \eqref{eq:residue} of the current $U$ defined by \eqref{eq:Udef},
    and $R'$ is almost semi-meromorphic with values in $\Hom(F^0,F)$.
    If $\codim \mathcal{F} > 0$, then $R' = 0$.
\end{proposition}

The proof of this proposition in turn relies on the following lemma, whose proof we postpone.

\begin{lemma} \label{lmaQ}
    Let $Q$ be the smooth form defined on $X\setminus Z$ by $Q \defeq \id -D\sigma$. Then $Q$ and $\nabla Q$ take values in
    $\Hom(F^0,F)$ and
    \begin{equation} \label{eq:nablau}
        \nabla u = \id - Q+u\nabla Q.
    \end{equation}
    Furthermore, if $\codim \mathcal{F} > 0$, then $Q = 0$.
\end{lemma}

\begin{proof}[Proof of Proposition~\ref{prop:Rasm}]
    On $X\setminus Z$, $U=u$, and consequently it follows by \eqref{eq:Rdef}
    and \eqref{eq:nablau} that $R|_{X\setminus Z} = Q-u\nabla Q$.
    Since almost semi-meromorphic currents have the SEP,
    $R'$ must be the almost semi-meromorphic extension of $Q-u\nabla Q$, provided it exists.
    By Proposition-Definition~\ref{prop:URdef}, $u$ has an extension to $X$ as an almost semi-meromorphic current, and
    by a similar argument, the same holds for $Q$ since almost semi-meromorphic currents are closed under multiplication by smooth forms.
    By Proposition~\ref{prop:dbarAsmExtension}, also $\nabla Q$ admits an almost semi-meromorphic extension, and since
    almost semi-meromorphic currents form an algebra, we conclude that $Q-u\nabla Q$ admits such an extension as well.
    By \eqref{eq:asmExtension}, we thus have that
    \begin{equation} \label{eq:Rprim}
        R' \defeq \lim_{\epsilon \to 0} \chi_\epsilon(Q-u\nabla Q).
    \end{equation}
    By Lemma~\ref{lmaQ}, it then follows that $R'$ takes values in $\Hom(F^0,F)$.
    Thus, by \eqref{eq:Udef}, \eqref{eq:Rdef}, \eqref{eq:nablau}, \eqref{eq:Rprim},
    and \eqref{eq:residue}, it follows that
    \[
        R-R' = \lim_{\epsilon \to 0} \dbar\chi_\epsilon \wedge  u = R(U).
    \]
    Finally, if $\codim \mathcal{F} > 0$, then $Q=0$ by the last part of Lemma~\ref{lmaQ},
    and consequently, $R' = 0$ by \eqref{eq:Rprim}.
\end{proof}

In the proof of Theorem~\ref{thm:R}, we will make use of singularity subvarieties associated
with a twisted resolution, which is a straight-forward generalization of the corresponding
subvarieties considered for a global resolution in for example \cite{AW1}*{Section~3}.
For a fixed $\alpha$, let $a^0_{\alpha,k}$ denote the part of $a_\alpha^0$ going between
$F_\alpha^{-k} \to F_\alpha^{-(k-1)}$.
Let $Z_\alpha^k \subseteq \opens_\alpha$ denote the analytic set where $a^0_{\alpha,k}$ does not have optimal rank,
i.e., $Z_\alpha^k$ is the set of $x$ such that the pointwise rank of $a_{\alpha,k}^0(x)$ is lower than its generic rank $r_k$.
Alternatively, $Z_\alpha^k$ is the analytic subset defined by the $r_k\times r_k$-minors of $a^0_{\alpha,k}$.

The $Z_\alpha^k$ are independent of the choice of resolution,
since they coincide (locally) with the corresponding sets defined for any minimal resolution.
It follows that $Z^k = \cup_\alpha Z_\alpha^k$ is an analytic subset of $X$.
Note that $Z^1=Z$, where $Z$ is defined above as the smallest analytic subset such that $\mathcal{F}|_{X\setminus Z}$ is a vector bundle.

These subvarieties may also be defined in a slightly different way, which allows for connections to local algebra.
Given a ring $R$ and a morphism $\varphi$ of free $R$-modules, we may consider its so-called \emph{Fitting ideal}
$I(\varphi)$, which is the ideal of $r\times r$-minors of $\varphi$, where $r = \rank \varphi$ is the largest size of
a non-vanishing minor.
By the alternative description of $Z_\alpha^k$ above, if $x \in \opens_\alpha$, it follows that
\begin{equation} \label{eq:Zkgerms}
	(Z^k_\alpha)_x = Z(I((a^0_{\alpha,k})_x)),
\end{equation}
where $( )_x$ denotes germs at $x$.

Since $(F_\alpha,a^0_\alpha)$ is a locally free resolution of $\mathcal{F}|_{\opens_\alpha}$, if
$x \in \opens_\alpha$, then $((F_{\alpha})_x,(a^0_\alpha)_x)$ is a free resolution of $\mathcal{F}_x$.
By the Buchsbaum-Eisenbud criterion for exactness, \cite{Eis}*{Theorem~20.9}, $\depth I((a^0_{\alpha,k})_x) \geq k$,
and since $\holo_{X,x}$ is Cohen-Macaulay, this is equivalent to that $\codim Z(I((a^0_{\alpha,k})_x)) \geq k$.
Hence, using \eqref{eq:Zkgerms}, it follows that
\[
\codim Z_\alpha^k \geq k.
\]
In a similar way, \cite{Eis}*{Corollary 20.12} translates into the inclusion
\[
    Z_\alpha^{k+1} \subseteq Z_\alpha^k.
\]
It follows that
\begin{equation} \label{eq:codimZk}
    \codim Z^k \geq k,
\end{equation}
and
\begin{equation} \label{eq:Zkinclusion}
    Z^{k+1} \subseteq Z^k.
\end{equation}
Note that $(\sigma^0)^\ell_{\ell+1}$ is smooth outside of $Z^{\ell+1}$.
Thus, it follows by \eqref{eq:sigmadef} and \eqref{eq:Zkinclusion} that
\begin{equation} \label{eq:sigmaSmooth}
    \sigma^\ell_\bullet \text{ is smooth outside of $Z^{\ell+1}$.}
\end{equation}

\begin{proof}[Proof of Theorem~\ref{thm:R}]
We begin by proving part (a), i.e., that $R$ takes values in $\Hom(F^0,F)$,
which is equivalent to proving that $R^\ell_\bullet = 0$ for $\ell \geq 1$.
By Proposition~\ref{prop:Rasm}, $R^\ell_\bullet = R(U)^\ell_\bullet$ for $\ell \geq 1$.
We will prove that $R(U)^\ell_\bullet = 0$ for $\ell=1$, the case $\ell>1$ follows in the same way.
By \eqref{eq:udef} and \eqref{eq:Udef}, it follows that $R(U)^1_\bullet = \sum_{m \geq 1} R([\sigma(\dbar\sigma)^{m-1}])^1_\bullet$. Here and throughout the proof we use the brackets to emphasize that we are considering the almost semi-meromorphic extension.
Thus, it suffices to prove that $R([\sigma(\dbar\sigma)^{m-1}])^1_\bullet = 0$ for $m \geq 1$, which we will prove by induction over $m$.
Applying $\dbar$ to $\sigma^2=0$ and using that $\sigma$ has degree $1$, we get that $0 = \dbar(\sigma^2) = (\dbar\sigma)\sigma-\sigma\dbar\sigma$, so $\sigma(\dbar\sigma)^{m-1} = (\dbar\sigma)^{m-1} \sigma$,
and we may thus equivalently prove that
\begin{equation} \label{eq:R1Induction}
    R([(\dbar\sigma)^{m-1}\sigma])^1_\bullet = 0
\end{equation}
for $m \geq 1$.
By \eqref{eq:sigmaSmooth}, $\sigma^1_\bullet$ is smooth outside of $Z^2$, so $\supp R([\sigma])^1_\bullet \subseteq Z^2$ by \eqref{eq:resSupport}.
Thus, \eqref{eq:R1Induction} holds in the basic case $m=1$ by the dimension principle, Proposition~\ref{prop:dimPrinciple}, and \eqref{eq:codimZk}
since $R([\sigma])^1_\bullet$ has bidegree $(0,1)$.

We may thus assume by induction that \eqref{eq:R1Induction} holds for $m$.
Note that since $\sigma^0$ takes values in $\bigoplus_\ell \Hom(F^{-\ell},F^{-(\ell+1)})$,
and $a^{\geq 1}$ takes values in $\bigoplus_{\ell \leq k} \Hom(F^{-\ell},F^{-k})$, it follows that
$\sigma$ takes values in $\bigoplus_{\ell < k} \Hom(F^{-\ell},F^{-k})$.
In particular, $R([(\dbar\sigma)^{m-1}\sigma])^1_\bullet$ takes values in
$\bigoplus_{k \geq m+1} \Hom(F^{-1},F^{-k})$.
In view of \eqref{eq:sigmaSmooth} and \eqref{eq:residueSmooth}, we thus have outside of $Z^{m+2}$ that
\[ R([(\dbar\sigma)^m\sigma])^1_\bullet = \sum_{\ell \geq m+1} (\dbar\sigma)^\ell_\bullet R([(\dbar\sigma)^{m-1}\sigma])^1_\ell. \]
By the induction assumption, this current vanishes outside of $Z^{m+2}$, and since
$R([(\dbar\sigma^m)\sigma])^1_\bullet$ is a pseudomeromorphic $(0,m+1)$-current, with support on $Z^{m+2}$,
which has codimension $\geq m+2$ by \eqref{eq:codimZk}, it vanishes by the dimension principle.
\medskip

We now prove part (c). By part (a), $R^\ell_\bullet=0$ for $\ell \geq 1$ and since $\nabla R = 0$ it follows that $R a^0 = 0$
since \[ 0 = (\nabla R)^1_\bullet =  a R^1_\bullet + R^1_\bullet a^1 -R^0_\bullet a^0. \]

It remains to prove parts (b) and (d).
By Proposition~\ref{prop:Rasm}, $(R')^\ell_\bullet = 0$ for $\ell \geq 1$.
In addition, the conditions $(R')^0_k = 0$ for $k < \codim \mathcal{F}$
and that $R'$ has support on $\supp \mathcal{F}$ are non-trivial only when $\codim \mathcal{F}> 0$,
but then $R'=0$ by Proposition~\ref{prop:Rasm}, and hence these conditions are indeed trivially verified.
Thus, it remains to prove parts (b) and (d), assuming that $R=R(U)$.

To prove that $\supp R(U) \subseteq \supp \mathcal{F}$, we may without loss of generality
assume that $\codim \mathcal{F} > 0$, so that $\supp \mathcal{F}=Z=Z^1$.
Since $\sigma^0$ and consequently $\sigma$ and $U$ are smooth outside of $Z$,
it then follows by \eqref{eq:resSupport} that indeed $\supp R(U) \subseteq Z = \supp \mathcal{F}$.
Since $R(U)^0_k$ is a sum of pseudomeromorphic $(0,k')$-currents with $k' \leq k$, with support on $Z$, it follows by the dimension
principle that $R(U)^0_k =0$ for $k < \codim \mathcal{F}$.
\end{proof}
Let us now return to the proof of Lemma~\ref{lmaQ}, which relies on the following technical lemma.
\begin{lemma} \label{lmaQ0}
    Define $Q^0 \defeq \id-a^0\sigma^0-\sigma^0 a^0$.
    Then $Q^0$ and $DQ^0$ take values in $\Hom(F^0,F)$,
    \begin{equation} \label{eq:Q0a0}
        Q^0 a^0 = 0, \quad DQ^0 a^0 = 0,
    \end{equation}
    and
    \begin{equation} \label{eq:sigma0Q0}
            \sigma = \sigma^0(D\sigma^0+Q^0)^{-1}.
    \end{equation}
    Furthermore, the element $Q = \id - D\sigma$ defined in Lemma~\ref{lmaQ} equals
    \begin{equation} \label{eq:QQ0}
        Q = Q^0-\sigma DQ^0.
    \end{equation}
\end{lemma}

In the proofs of Lemma~\ref{lmaQ} and Lemma~\ref{lmaQ0},
we will use that if
\[
	\alpha,\beta \in C^\bullet(\cover',\currs^{0,\bullet}(\Homs(F,F))),
\]
and $\alpha$ takes values in $\Hom(F^0,F)$ and $\beta$ is nilpotent and takes values in $\bigoplus_{k \geq 1} \Hom(F,F^{-k})$, then
$\alpha\beta = 0$, and
\begin{equation} \label{eq:alphaidbeta}
    \alpha(\id+\beta)^{-1} = \alpha(\id-\beta+\beta^2-\dots) = \alpha.
\end{equation}

\begin{proof}[Proof of Lemma~\ref{lmaQ0}]
It follows from the properties defining $\sigma^0$ that
\[
	Q^0\in C^0(\cover',\currs^{0,0}(\Homs(F^0,F^0))),
\]
and it may alternatively be defined by $Q_\alpha^0 = \id_{F_\alpha^0}- (a_\alpha^0 \sigma_\alpha^0)^0_0$,
and it has the property that $Q^0 a^0 = 0$.
Since $a = a^0 + a^{\geq 1}$, we thus have that
\begin{equation} \label{eq:Dsigma0}
	D \sigma^0 =
	a \sigma^0 + \sigma^0 a =
	\id - Q^0 + a^{\geq 1} \sigma^0 + \sigma^0 a^{\geq 1}.
\end{equation}
Since $(\sigma^0)^2 = 0$, it follows from \eqref{eq:sigmadef} that \eqref{eq:sigma0Q0} holds.
Thus,
\begin{equation}
\begin{aligned} \label{eq:Dsigma}
	D \sigma &=
	D \sigma^0 (D \sigma^0 + Q^0)^{-1} - \sigma^0 D (D \sigma^0 + Q^0)^{-1} \\ &=
	(D \sigma^0 + Q^0 - Q^0)(D \sigma^0 + Q^0)^{-1} - \sigma^0 D (D \sigma^0 + Q^0)^{-1} \\ &=
	\id - Q^0 - \sigma^0 D (D \sigma^0 + Q^0)^{-1},
\end{aligned}
\end{equation}
where we have used \eqref{eq:Dderivation} in the first equality, and in the third equality,
we have used that $Q^0 (D\sigma^0+Q^0)^{-1} = Q^0$ by \eqref{eq:alphaidbeta} since $D\sigma^0+Q^0=\id+a^{\geq 1}\sigma^0+\sigma^0 a^{\geq 1}$ by \eqref{eq:Dsigma0}.

Since $Q^0$ takes values in $\Hom(F^0,F^0)$ and $Q^0 a^0 = 0$,
$D Q^0 = a^{\geq 1} Q^0 - Q^0 a^1$, and in particular, $D Q^0$
takes values in $\Hom(F^0,F)$. Since $a^1 a^0 = a^0 a^1$ by \eqref{eq:twisted1}, it follows that
$DQ^0 a^0=0$, so \eqref{eq:Q0a0} is proved.
Using that $D\alpha^{-1}=-\alpha^{-1}(D\alpha) \alpha^{-1}$, we also get that
\begin{equation}
\begin{aligned} \label{eq:DDsigma0Q0inv}
	D (D \sigma^0 + Q^0)^{-1} &=
	-(D \sigma^0 + Q^0)^{-1} D Q^0 (D \sigma^0 + Q^0)^{-1} \\ &=
	-(D \sigma^0 + Q^0)^{-1} D Q^0,
\end{aligned}
\end{equation}
where we in the second equality used that $D Q^0 (D \sigma^0 + Q^0)^{-1} = D Q^0$,
for the same reasons as for $Q^0$ in the previous paragraph.
The formula \eqref{eq:QQ0} then follows by combining \eqref{eq:sigma0Q0}, \eqref{eq:Dsigma}
and \eqref{eq:DDsigma0Q0inv}.
\end{proof}

\begin{proof}[Proof of Lemma~\ref{lmaQ}]
Since $Q^0$ and $DQ^0$ take values in $\Hom(F^0,F)$, it follows by \eqref{eq:QQ0} that also $Q$ takes values
in $\Hom(F^0,F)$. By \eqref{eq:Q0a0} and \eqref{eq:QQ0} it follows that $Q a^0 = 0$,
and hence also $\nabla Q$ takes values in $\Hom(F^0, F)$.
Note that $Q^0_\alpha$ is the orthogonal projection onto the orthogonal complement of $\im a^0_\alpha : F_\alpha^{-1} \to F_\alpha^0$.
If $\codim \mathcal{F} > 0$, it follows that $Q^0$ vanishes.
Hence, in that case $Q$ vanishes as well by \eqref{eq:QQ0}.

It remains to prove \eqref{eq:nablau}.
Since by definition, $Q=\id-D\sigma$, we get that
\begin{equation} \label{eq:iddbarsigma}
\id-\dbar\sigma=\nabla\sigma+Q.
\end{equation}
Thus, by \eqref{eq:udef},
\begin{equation} \label{eq:uAlternative}
    u = \sigma(\nabla\sigma+Q)^{-1}.
\end{equation}
It then follows that
\begin{align} \label{eq:nablasigmaQinv}
	\nabla(\nabla \sigma + Q)^{-1} &=
	-(\nabla \sigma + Q)^{-1} \nabla Q (\nabla \sigma + Q)^{-1} =
	-(\nabla \sigma + Q)^{-1} \nabla Q,
\end{align}
where we in the second equality have used that $\nabla Q(\nabla\sigma+Q)^{-1}=\nabla Q$
by \eqref{eq:alphaidbeta} and \eqref{eq:iddbarsigma}.
Using \eqref{eq:uAlternative} and \eqref{eq:nablasigmaQinv}, we get that
\begin{align*}
	\nabla u &=
	\nabla \sigma (\nabla \sigma + Q)^{-1} -
	\sigma \nabla(\nabla \sigma + Q)^{-1}
	\\ &=
	(\nabla \sigma + Q - Q)(\nabla \sigma + Q)^{-1} +
	u \nabla Q
	\\ &=
	\id - Q + u \nabla Q,
\end{align*}
where we in the third equality have used that $Q(\nabla\sigma+Q)^{-1}=Q$ by
\eqref{eq:alphaidbeta} and \eqref{eq:iddbarsigma}.
\end{proof}

\begin{remark}
    In the case that $\mathcal{F}$ has a global locally free resolution of finite length $(E,\varphi)$,
    as in Example~\ref{ex:globalRes}, then $\sigma=\sigma^0$ since $(\sigma^0)^2=0$.
    Assuming additionally that $\codim \mathcal{F} > 0$, the forms and currents $\sigma$, $u$, $U$
    and $R$ then define global $\End(E)$-valued currents, which coincide with the ones
    defined in \cite{AW1}, and the calculations here reduce to calculations similar to the ones
    in \cite{AW1}*{Section~2}, since then $Q^0$, $Q$ and $R'$ vanish.

    In \cite{AW1}, a way of handling the case when $\codim \mathcal{F} = 0$ is only mentioned briefly by a comment
    in \cite{AW1}*{Corollary~2.4}, while the remainder of the paper is written under the assumption that
    $\codim \mathcal{F} > 0$. This is a natural assumption in the context of that paper,
    and it is not really specified how $R$ is defined in the case $\codim \mathcal{F} = 0$.
    For us, it is natural to not have any assumptions on the codimension, so we have
    made sure to provide a definition which works also this case.
    The main issue is that in \cite{AW1}*{Section~2}, the equality
    which in our notation may be written as $a^0 \sigma^0 + \sigma^0 a^0 = \id$ is being used,
    but this equality only holds when $\codim \mathcal{F} > 0$, i.e., when $Q^0=0$.
    We then use different formulas to define $u$ and $R$ in \eqref{eq:udef} and \eqref{eq:Rdef}
    than in \cite{AW1}. The formulas we use as a definition actually appear as alternative formulas
    in \cite{AW1}, since they are equivalent when $\codim \mathcal{F} > 0$.
    Our definitions are made so that $R$ gets the crucial properties
    that it takes values in $\Hom(F^0,F)$ and is $\nabla$-cohomologous to the identity mapping.

    In the case that $\codim \mathcal{F} = 0$, with a global resolution $(E,\varphi)$,
    then $Q_0 \neq 0$, and using that $a^1$ is the identity, it follows that $DQ_0 = 0$,
    so by \eqref{eq:QQ0}, $Q=Q_0$, which outside of $Z$ equals the orthogonal projection onto
    the orthogonal complement of $\im \varphi_1$. By Proposition~\ref{prop:Rasm} and \eqref{eq:Rprim}, it follows that
    \[
        R = R(U)+[Q_0]+[u\dbar Q_0],
    \]
    where we by $[Q_0]$ and $[u \dbar Q_0]$ denote the almost semi-meromorphic extensions.
    The duality result \cite{AW1}*{Corollary~2.4} states that a section $\phi$ of $E_0$
    belongs to $\im \varphi_1$ if and only if $\varphi_1$ belongs generically to
    $\im \varphi_1$ and $R(U)\phi = 0$.
    Note that $[Q_0]$ is the only part of $R$ with values in $\Hom(F^0,F^0)$,
    and that $[Q_0]\phi = 0$ if and only if $\phi$ belongs generically to $\im \varphi_1$,
    and it holds that if $[Q_0] \phi = 0$, then $[u\dbar Q_0] \phi = 0$. Thus, with our definition of $R$,
    it follows that \cite{AW1}*{Corollary~2.4} may be reformulated succinctly as:
    $\phi$ belongs to $\im \varphi_1$ if and only if $R\phi = 0$.
\end{remark}
\begin{example}[Explicit description of $R^0$ and $R^1$]
Assume that $(F,a)$ is a twisted resolution of $\mathcal{F}$, and that $\codim \mathcal{F} > 0$.

We consider first the components $R^0_\alpha$ of \v{C}ech degree $0$.
Since $\sigma_\alpha = \sigma^0_\alpha$ equals the form $\sigma$ defined in \cite{AW1} associated with $(F_\alpha,a^0_\alpha)$
on $\opens_\alpha$, and comparing the definitions in this section with the ones in \cite{AW1}, it follows that
$u_\alpha^0$, $U_\alpha^0$ and $R_\alpha^0$ equal the corresponding forms and currents $u$, $U$, and $R$ defined in \cite{AW1}*{Section~2}
associated with $(F_\alpha,a^0_\alpha)$.

We consider next the components $R^1_{\alpha\beta}$ of \v{C}ech degree $1$.
Using that $\nabla R = 0$ by Proposition-Definition~\ref{prop:URdef}, and considering a component of \v{C}ech degree $1$,
and using that $R a^0 = 0$, we obtain that
\begin{equation*}
    0 = (\nabla R)_{\alpha \beta} =
    a_\alpha^0 R_{\alpha \beta}^1 + a^1_{\alpha\beta} R^0_\beta -
    R^0_\alpha a^1_{\alpha\beta} - \dbar R^1_{\alpha\beta},
\end{equation*}
i.e.,
\begin{equation} \label{eq:R1comparisonFormula}
    R^0_\alpha a^1_{\alpha\beta} - a^1_{\alpha\beta} R^0_\beta = a^0_\alpha R^1_{\alpha\beta} - \dbar R^1_{\alpha\beta}.
\end{equation}
Note that $a^1_{\alpha\beta} : (F_\beta,a^0_\beta) \to (F_\alpha,a^0_\alpha)$ is a morphism of complexes, and \eqref{eq:R1comparisonFormula}
in fact coincides with the comparison formula from \cite{Lar} in this situation. Indeed, as explained above, $R^0_\alpha$ and $R^0_\beta$ are
the residue currents associated with $(F_\alpha,a^0_\alpha)$ and $(F_\beta,a^0_\beta)$, respectively.
Using that $(\sigma^0_\alpha)^2 = 0$, it follows that $\sigma^0_\alpha(\dbar\sigma^0_\alpha)^m\sigma^0_\alpha = 0$ for any $m \geq 1$,
and then a calculation using \eqref{eq:sigmadef} and \eqref{eq:udef} yields that if we let $u_\alpha^0 = \sigma^0_\alpha + \sigma^0_\alpha\dbar\sigma^0_\alpha+\dots$,
then $U_{\alpha\beta}^1 = u_\alpha^0 a^1_{\alpha\beta} u_\beta^0$. Since $R^1_{\alpha\beta} = R(U_{\alpha\beta}^1)$ by Proposition~\ref{prop:Rasm},
it follows that $R^1_{\alpha\beta}$ equals the current $M$ defined in \cite{Lar}*{(3.1)} associated with the morphism of complexes $a^1_{\alpha\beta}$,
and \eqref{eq:R1comparisonFormula} then coincides with the comparison formula \cite{Lar}*{Theorem~3.2} (see in particular \cite{Lar}*{(3.4)}).
\end{example}

\section{Twisted resolutions and residue currents}
\label{section:main-theorem}
Let $(F,a)$ be a twisted resolution of a coherent $\holo_X$-module $\mathcal{F}$, and let $G$ be a holomorphic vector bundle. Note that the operator $D$ on $C^\bullet(\cover,\Homs^\bullet(F,G))$ can be decomposed as $D = D^0 + D^{\geq 1}$, where
\[
	D^0: C^p(\cover,\Homs^r(F,G)) \to
	C^p(\cover,\Homs^{r+1}(F,G)),
\]
and
\[
	D^{\geq 1}: C^p(\cover,\Homs^r(F,G)) \to
	\bigoplus_{j \geq 1} C^{p+j}(\cover,\Homs^{r-j+1}(F,G)).
\]
The triple $(C^\bullet(\cover,\Homs^\bullet(F,G)),D^0,D^{\geq 1})$ is referred to as a \emph{twisted complex}, and it was shown in \cite{TT1}*{Theorem~2.9} that
\[
	H^k\left(
	\bigoplus_{p+r=\bullet}
	C^p(\cover,\Homs^r(F,G))
    \right) \cong
	\Ext^k(\mathcal{F},G).
\]
Our objective now is to make this isomorphism explicit by representing $\Ext^k(\mathcal{F},G)$ as the $k$th cohomology of $\Hom(\mathcal{F},\currs^{0,\bullet}(G))$ as explained in the introduction.
We will begin by defining the operator $v$ that appears in the statement of Theorem~\ref{thm:main}.

\subsection{The operator $v$} \label{ssection:vdef}

Let $(\rho_\alpha)$ be a partition of unity subordinate to $\cover$. We define an operator $v: C^p(\cover,\currs^{0,q}(\Homs^r(F,G))) \to C^{p-1}(\cover,\currs^{0,q}(\Homs^r(F,G)))$ as
\[
	(v f)_{\alpha_0 \dots \alpha_{p-1}} \defeq
	\sum_{\alpha} \rho_\alpha
	f_{\alpha \alpha_0 \dots \alpha_{p-1}}
\]
for $p \geq 1$ and $vf \defeq 0$ otherwise. Here $\rho_\alpha f_{\alpha \alpha_0 \dots \alpha_{p-1}}$, which is defined on $\opens_{\alpha \alpha_0 \dots \alpha_{p-1}}$, is extended by 0 to $\opens_{\alpha_0 \dots \alpha_{p-1}}$.

A calculation yields that if $f \in C^p(\cover,\currs^{0,q}(\Homs^r(F,G)))$, then
\begin{equation*}
	Dvf = f - vDf
\end{equation*}
if $p \geq 1$ and $Dvf=0$ otherwise. Moreover, for $j \geq 1$, it follows by induction
and the fact that $D\dbar=-\dbar D$ that
\begin{equation*}
	Dv(\dbar v)^j f =
	(\dbar v)^j f + v(\dbar v)^{j-1} \dbar f -
	v(\dbar v)^j Df
\end{equation*}
if $p \geq j+1$ and $D v(\dbar v)^j f=0$ otherwise.
By $(\dbar v)^j f$, we mean the composition $\dbar\circ v$ applied $j$ times to $f$.

Let now $f \in C^\bullet(\cover,\currs^{0,\bullet}(\Homs^\bullet(F,G)))$ be an arbitrary element, and as above we let $f^j$ denote the parts of $f$ of \v{C}ech degree $j$. By computing $Dv(f-f^0)$ and $Dv(\dbar v)^j \left(f - \sum_{\ell=0}^j f^\ell \right)$, we obtain the formulas
\begin{equation}
\label{eq:Dv1}
	Dvf = f - vDf - f^0 + v D f^0,
\end{equation}
and
\begin{equation}
\label{eq:Dv2}
\begin{aligned}
	Dv(\dbar v)^j f &=
	(\dbar v)^j f + v(\dbar v)^{j-1} \dbar f -
	v(\dbar v)^j Df
    \\ &- (\dbar v)^j f^j - v(\dbar v)^{j-1} \dbar f^j +
	\sum_{\ell=0}^j \left( v(\dbar v)^j D f^\ell \right)
\end{aligned}
\end{equation}
for any $j \geq 1$.

\smallskip

We denote by $\iota$ the inclusion
\[
	\Hom(\mathcal{F},\currs^{0,q}(G)) \stackrel{\cong}{\to} H^0( \oplus_{p+r=\bullet} C^p(\cover,\Homs^r(F,\currs^{0,q}(G)))) \to C^0(\cover,\Homs^r(F,\currs^{0,q}(G))).
\]
Note that if $f' \in \Hom(\mathcal{F},\currs^{0,q}(G))$, then $\iota f'$ takes values in $\Hom(F^0,G)$ and $D (\iota f') = \iota f' a^0 = 0$.

Given $f \in C^0(\cover,\Homs(F,\currs^{0,\bullet}(G)))$ that takes values in $\Hom(F^0,G)$ and is such that $f a^0 = 0$, we define
$v_0 f \in \Hom(\mathcal{F},\currs^{0,\bullet}(G))$ by $v_0 f := \sum \rho_\alpha \tilde{f}_\alpha$, where $\tilde{f}_\alpha$ denotes the
morphism in $\Homs(\mathcal{F},G)(U_\alpha)$ determined by $f_\alpha$ induced by the surjection $F^0_\alpha \to \mathcal{F}|_{U_\alpha}$. Indeed,
$\tilde{f}_\alpha$ is well-defined since $f a^0 = 0$.
Unwrapping the definitions, it follows that
\begin{equation*}
    (\iota v_0 f)_{\alpha_0} = \sum_{\alpha} \rho_\alpha f_\alpha a^1_{\alpha\alpha_0} .
\end{equation*}
By a similar calculation as above,
\begin{equation} \label{eq:Dv1prim}
	\iota v_0 f = f - vD^1f.
\end{equation}

If $f$ is furthermore $D^1$-closed, $v_0 f$ is independent of the choice of partition of unity $(\rho_\alpha)$, and $v_0 f|_{U_\alpha} = \tilde{f}_\alpha$.
In that case, $f$ is in fact $D$-closed, and $v_0 f$ is the element corresponding to $f$ through the inverse isomorphism
\[
	H^0( \oplus_{p+r=\bullet} C^p(\cover,\Homs^r(F,\currs^{0,q}(G)))) \stackrel{\cong}{\to} \Hom(\mathcal{F},\currs^{0,q}(G)).
\]

\subsection{Cohomology of graded complexes}

In the proof of Theorem~\ref{thm:main}, we will make repeated use of
a basic result about cohomology of certain graded complexes.
Graded complexes generalize double complexes, and we will make use
of this result both for double complexes and for twisted complexes.

Let $(A^{p,q},d_{p,q,r})$ be a \emph{graded complex} of abelian groups, i.e., for $p,q \in \Z$, $A^{p,q}$ are abelian groups,
and for any $r \in \Z$  there are morphisms
$d_{p,q,r} : A^{p,q} \to A^{p+r,q+1-r}$ such that if
\[ A^k \defeq \bigoplus_{p+q=k} A^{p,q}, \quad d_k \defeq \bigoplus_{\substack{p+q=k \\ r\in \Z}} \, d_{p,q,r} : A^k \to A^{k+1},\]
then $(A^\bullet,d)$ defines a complex.
We will always assume that $A^{p,q}=0$ for $p<0$ or $q<0$, and furthermore
that $d_{p,q,r} = 0$ for $r < 0$.
Note that a graded complex defines a filtered complex with the filtration given
by
\[
	F^j A^k = \bigoplus_{\substack{p+q=k \\ p \geq j}} A^{p,q}.
\]
We write $d^0 = \bigoplus_{p,q} d_{p,q,0}$ and $d' = \bigoplus_{r \geq 1} d_{p,q,r}$.
Since $d_{p,q,r} = 0$ for $r<0$, it follows that $d=d^0+d'$ and
$(d^0)^2 = 0$, so for a fixed $p$, $(A^{p,\bullet},d^0)$ defines a complex.

\begin{proposition} \label{prop:cohomologyGradedComplex}
    Let $(A^{p,q},d_{p,q,r})$ be a graded complex of abelian groups, such
    that $A^{p,q}= 0$ if $p<0$ or $q<0$, and $d_{p,q,r} = 0$ for $r < 0$.
    Assume furthermore that $H^q(A^{p,\bullet},d^0) = 0$ for $q>0$.
    Then the inclusion $A^{k,0} \hookrightarrow A^k$ induces an isomorphism
    \[
        H^k(H^0(A^{\bullet,\bullet},d^0),d') \overset{\cong}{\to} H^k(A^\bullet,d).
    \]
\end{proposition}

This follows from the general theory of spectral sequences associated with a filtered complex, see for example \cite{Gun3}*{Theorem C.8},
but may also be proven directly by arguments involving straight-forward diagram chases.

A special case of a graded complex is a double complex $(A^{p,q},d',d'')$ i.e., when $d_{p,q,r} = 0$ for $r \neq 0,1$,
so $d^0=d'' : A^{p,q} \to A^{p,q+1}$, $d' : A^{p,q} \to A^{p+1,q}$ and $d'd'' = - d''d'$.
In this case, one may construct a new double complex with the same total complex by switching roles of the indices.
Applying Proposition~\ref{prop:cohomologyGradedComplex}, one obtains the following well-known consequence:
Let $(A^{p,q},d',d'')$ be a double complex and assume that
$H^p(A^{\bullet,q},d') = 0$ for $p > 0$ and that $H^q(A^{p,\bullet}) = 0$ for $q > 0$.
Then there are isomorphisms
\begin{equation} \label{eq:doubleComplexCohomologyIsomorphism}
    H^k(H^0(A^{\bullet,\bullet},d''),d') \cong H^k(A^\bullet,d) \cong
    H^k(H^0(A^{\bullet,\bullet},d'),d'').
\end{equation}
Given an element $[[x']] \in H^k(H^0(A^{\bullet,\bullet},d''),d')$, one may find an element $y \in A^{k-1}$ such that
\begin{equation} \label{eq:dPotential}
	d y = x' - x'',
\end{equation}
for some $x'' \in A^{0,k}$. Then, the composed isomorphism between the left-most and right-most group in \eqref{eq:doubleComplexCohomologyIsomorphism} is given by
\begin{equation} \label{eq:doubleComplexExplicitIsomorphism}
    [[x']] \mapsto [[x'']].
\end{equation}

\subsection{Proof of Theorem~\ref{thm:main}}

In the proof of Theorem~\ref{thm:main} we will make use of the following
useful result, which is essentially a consequence of the stalkwise injectivity
of currents by Malgrange.

\begin{lemma} \label{lma:extCurrentsVanish}
    Let $\mathcal{F}$ be a coherent $\holo_X$-module, and let
    $G$ be a locally free $\holo_X$-module. Then
    \[ \Ext^k(\mathcal{F},\currs^{0,q}(G)) = 0 \]
    for $k \geq 1$.
\end{lemma}

In particular, if $\mathcal{F}$ is coherent, then
$\currs^{0,\bullet}(G)$ is an acyclic resolution of $G$
for the functor $\Hom(\mathcal{F},\bullet)$, so
\begin{equation}
    \Ext^k(\mathcal{F},G) \cong H^k(\Hom(\mathcal{F},\currs^{0,\bullet}(G))).
\end{equation}

\begin{proof}
    By \cite{Mal}*{Theorem~VII.2.4}, $\currs^{0,q}$ is stalkwise injective,
    and consequently $\currs^{0,q}(G)$ is as well.
    Thus,
    \[ \Ext^\ell_{\holo_{X,x}}(\mathcal{F}_x,\currs^{0,q}(G)_x) = 0 \]
    for $\ell \geq 1$.
    Since $\mathcal{F}$ is coherent,
    \[
        \Exts^{\ell}(\mathcal{F},\currs^{0,q}(G))_x \cong \Ext^\ell_{\holo_{X,x}}(\mathcal{F}_x,\currs^{0,q}(G)_x),
    \]
    cf., e.g., \cite{Har}*{Proposition~III.6.8}\footnote{The result in \cite{Har} is stated for noetherian schemes,
    but essentially the same proof translates into our setting of complex manifolds. Indeed, the statement is local, so we may replace $X$
    by a small neighborhood where $\mathcal{F}$ admits a (locally) free resolution. The remainder of the proof proceeds as in \cite{Har}.}.
    In addition, since $\currs^{0,q}$ is fine, $\Homs(\mathcal{F},\currs^{0,q}(G))$ is as well,
    so
    \[ H^k(X,\Homs(\mathcal{F},\currs^{0,q}(G))) = 0. \]
    To conclude, $H^r(X,\Exts^s(\mathcal{F},\currs^{0,q}(G))) = 0$ for $r+s=k$,
    so $\Ext^k(\mathcal{F},\currs^{0,q}(G)) = 0$ by the local to global
    spectral sequence of $\Ext$.
\end{proof}

Before the proof of Theorem~\ref{thm:main}, we will also make some preliminary calculations.
Let $\eta \in C^\bullet(\cover,\currs^{0,\bullet}(\Homs^\bullet(F,G)))$ be an element such that $\eta a^0 = 0$ and that $\eta$ takes values in $\Hom(F^0,G)$.
Then $D\eta = D^1\eta$. In particular, it follows that
\begin{equation} \label{eq:DetaJ}
	D\eta^j = (D^1 \eta^j) = (D\eta)^{j+1}.
\end{equation}
By \eqref{eq:Dv1} and \eqref{eq:Dv2}, we have that
\begin{equation} \label{eq:DvDbarvEta0}
	D \sum_{j \geq 1} v(\dbar v)^{j-1} \eta^j = \sum_{j \geq 1} (\dbar v)^{j-1} \eta^j + \sum_{j \geq 2} v(\dbar v)^{j-2} \dbar \eta^j - \sum_{j \geq 1} v(\dbar v)^{j-1} D\eta^j.
\end{equation}
Thus, if $D\eta = D^1 \eta$, it follows from \eqref{eq:DetaJ} and \eqref{eq:DvDbarvEta0} that
\begin{equation} \label{eq:DvDbarvEta}
	D \sum_{j \geq 1} v(\dbar v)^{j-1} \eta^j = \sum_{j \geq 1} (\dbar v)^{j-1} \eta^j - \sum_{j \geq 1} v(\dbar v)^{j-1} (\nabla \eta)^{j+1}.
\end{equation}

\begin{proof}[Proof of Theorem~\ref{thm:main}]
We first prove that \eqref{eq:twisted-iso-formula} is a morphism of complexes, i.e., that
\begin{equation}
	-\dbar \sum_j v_0 (\dbar v)^j (\xi R)^j = \sum_j v_0 (\dbar v)^j ( (D\xi) R)^j.
\end{equation}
Since the morphism $\iota$ defined in Section~\ref{ssection:vdef} is injective, and $\dbar(\dbar v)^j = 0$ for $j \geq 1$,
it suffices by \eqref{eq:Dv1prim} to prove that
\begin{equation}
	-\dbar (\xi R)^0 + \dbar \sum_j v D^1 (\dbar v)^j (\xi R)^j = \sum_j (\dbar v)^j ((D \xi) R)^j
	- \sum_j v D^1(\dbar v)^j ( (D\xi) R)^j.
\end{equation}
Note that since $\dbar \xi = 0$ and $\nabla R = 0$, it follows that $(D \xi) R = \nabla (\xi R)$. In addition, since
$R a^0 = 0$, and $R$ takes values in $\Hom(F^0,G)$, using \eqref{eq:DetaJ}, it thus suffices to prove that
if $\eta \in C^\bullet(\cover,\currs^{0,\bullet}(\Homs^\bullet(F,G)))$ is such that $D\eta = D^1\eta$, then
\begin{equation} \label{eq:morphismOfComplexesEq}
	-\dbar \eta^0 + \dbar \sum_j v D (\dbar v)^j \eta^j = \sum_j (\dbar v)^j (\nabla \eta)^j
	-\sum_j v D(\dbar v)^j (\nabla \eta)^j.
\end{equation}
To prove this, we begin with the following calculation of the left-hand side,
\begin{align*}
	-\dbar \eta^0 + \dbar \sum_j v D (\dbar v)^j \eta^j = -\dbar \eta^0 + \dbar v D \eta^0 - \dbar v \dbar \sum_{j \geq 1} Dv(\dbar v)^{j-1} \eta^j = \\
	-\dbar \eta^0 + \dbar v (D\eta)^1 - (\dbar v)\dbar \eta^1 + \sum_{j \geq 1} (\dbar v)^{j+1} (\nabla \eta)^{j+1} = \sum_{j \geq 0} (\dbar v)^j (\nabla \eta)^j
\end{align*}
where we have used that $\dbar D = -D\dbar$ in the first equality, and \eqref{eq:DetaJ} and \eqref{eq:DvDbarvEta} in the second equality.
To prove \eqref{eq:morphismOfComplexesEq}, it thus remains to prove that the second term in the right-hand side vanishes, i.e.,
\begin{equation} \label{eq:finalTermVanishes}
	\sum_j v D(\dbar v)^j (\nabla \eta)^j = 0.
\end{equation}
By a similar calculation as above, we obtain that
\begin{align*}
	\sum_j D(\dbar v)^j (\nabla \eta)^j =  D(\nabla \eta)^0 - \dbar \sum_{j \geq 1} D v (\dbar v)^{j-1} (\nabla \eta)^j
	= \\
	(D(\nabla \eta))^1 - \dbar(\nabla \eta)^1 + \sum_{j \geq 1} (\dbar v)^j (\nabla(\nabla \eta))^{j+1}  = \sum_{j \geq 0} (\dbar v)^j (\nabla^2 \eta)^{j+1} = 0
\end{align*}
which proves \eqref{eq:finalTermVanishes}.

\medskip

It remains to prove that \eqref{eq:twisted-iso-formula} induces the isomorphism \eqref{eq:twisted-iso}.
Consider the double complex
\[
	E^{q,\ell} \defeq
	\bigoplus_{p+r=\ell}
	C^p(\cover,\currs^{0,q}(\Homs^r(F,G)))
\]
with differentials $D$ and $-\dbar$, and total differential $\nabla = D - \dbar$.

We have that
\begin{equation}
\begin{aligned} \label{eq:HqEql}
	H^q(E^{\bullet,\ell}) &=
	\bigoplus_{p+r=\ell} \prod_{(\alpha_0,\dots,\alpha_p)}
	H^{0,q}(\opens_{\alpha_0 \dots \alpha_p},
	\Homs^r(F_{\alpha_p},G)) \\
&= \left\{\begin{array}{cc} 0 & \text{ for $q > 0$,} \\
    \bigoplus_{p+r=\ell} C^p(\cover,\Homs^r(F,G)) & \text{ for $q=0$,}
    \end{array}
\right.
\end{aligned}
\end{equation}
where the case $q > 0$ in the second equality follows by Cartan's Theorem B, since $\opens_{\alpha_0 \dots \alpha_p}$ is Stein
and $\Homs^r(F_{\alpha_p},G)$ is coherent.

We claim that
\begin{equation} \label{eq:HlEql}
H^\ell(E^{q,\bullet}) = \left\{\begin{array}{cc} 0 & \text{ for $\ell > 0$,} \\
    \Hom(\mathcal{F}, \currs^{0,q}(G)) & \text{ for $\ell=0$.}
        \end{array}
\right.
\end{equation}
Thus, combining \eqref{eq:HqEql} and \eqref{eq:HlEql}, we obtain an isomorphism \eqref{eq:twisted-iso} by \eqref{eq:doubleComplexCohomologyIsomorphism}.

To prove \eqref{eq:HlEql} we fix $q$ and consider
\[
	A^{p,r} \defeq
	C^p(\cover,\currs^{0,q}(\Homs^r(F,G))) =
	\prod_{(\alpha_0,\dots,\alpha_p)}
	\Homs(F_{\alpha_p}^{-r}, \currs^{0,q}(G))
	(\opens_{\alpha_0 \dots \alpha_p}),
\]
which provides a decomposition $E^{q,\ell} = \bigoplus_{p+r=\ell} A^{p,r}$.
By Lemma~\ref{lma:extCurrentsVanish} and the long exact sequence of $\Ext$,
for $\opens$ open,
$\Homs(\bullet,\currs^{0,q}(G))(\opens)$
is an exact
functor on the category of coherent $\holo_{\opens}$-modules.
Thus,
\[
	H^r(A^{p,\bullet}) \cong
	\prod_{(\alpha_0,\dots,\alpha_p)}
	H^r(\Homs(F_{\alpha_p}^{-\bullet},\currs^{0,q}(G))(\opens_{\alpha_0 \dots \alpha_p})) = 0
\]
for $r > 0$. In addition,
\[
	H^0(A^{p,\bullet}) = \prod_{(\alpha_0,\dots,\alpha_p)}
	\Homs(\mathcal{F}, \currs^{0,q}(G))
	(\opens_{\alpha_0 \dots \alpha_p}).
\]
We may thus apply Proposition~\ref{prop:cohomologyGradedComplex} to compute $H^\ell(E^{q,\bullet})$.
A calculation yields that the differential induced by $D^{\geq 1}$ on $H_{D^0}^0(A^{\bullet,\bullet}) \cong C^\bullet(\cover,\Homs(\mathcal{F},\currs^{0,q}(G)))$
is just the ordinary \v{C}ech coboundary.
Thus
\[ H^\ell(E^{q,\bullet}) \cong H_{D^{\geq 1}}^\ell(H_{D^0}^0(A^{\bullet,\bullet})) \cong \check{H}^\ell(\cover,\Homs(\mathcal{F}, \currs^{0,q}(G))),\]
which vanishes for $\ell > 0$ since $\Homs(\mathcal{F}, \currs^{0,q}(G))$ is fine, and is isomorphic to $\Hom(\mathcal{F}, \currs^{0,q}(G))$ for $\ell = 0$,
which proves \eqref{eq:HlEql}.

We finally prove that the isomorphism is given by \eqref{eq:twisted-iso-formula}. To this end, let $\xi$ be a $D$-closed element of $\bigoplus_{p+r=k} C^p(\cover,\Homs^r(F,G))$.
Let $x' := \xi$ and $x'' := \sum_j (\dbar v)^j (\xi R)^j$. We claim that it suffices to find a $\nabla$-potential of $x' - x''$, i.e., that there exists a
$y \in \bigoplus_{p+r=\bullet} C^p(\cover,\Homs^r(F,\currs^{0,\bullet}(G)))$ such that \eqref{eq:dPotential} holds.
Indeed, $x''$ has \v{C}ech and Hom degree $0$, so it belongs to $E^{q,0}$. It then follows by \eqref{eq:dPotential} that $x''$ is
$D$-closed and $\dbar$-closed, and the class $[[x'']]$ corresponding to $x''$ in \eqref{eq:doubleComplexExplicitIsomorphism}
equals $v_0 x'' = \sum_j v_0 (\dbar v)^j (\xi R)^j$.

To find the $\nabla$-potential, we set $w = v + v (\dbar v) + v (\dbar v)^2 + \dots$, and we claim that
\begin{equation} \label{eq:nablaPotential}
	\nabla((-1)^k \xi U + w(\xi R)) =
	\xi - \sum_j (\dbar v)^j (\xi R)^j.
\end{equation}
To prove \eqref{eq:nablaPotential},
note first that since $\xi$ is $D$-closed and holomorphic, $\nabla \xi = 0$.
In addition, $R$ is $\nabla$-closed by Proposition-Definition~\ref{prop:URdef},
and thus it follows from \eqref{eq:nablaDerivation} that $\nabla (\xi R) = 0$, i.e., $D(\xi R) = \dbar (\xi R)$.
Since $R$ takes values in $\Hom(F^0,F)$ and $R a^0=0$ by Theorem~\ref{thm:R},
it follows that $D (\xi R)^j$ has \v{C}ech degree $j+1$, and hence for degree reasons it follows that $D (\xi R)^j = \dbar (\xi R)^{j+1}$.
Thus it follows from \eqref{eq:Dv1} that
\begin{equation} \label{eq:DvxiR1}
	Dv (\xi R) =
	\xi R - v \dbar(\xi R) - (\xi R)^0 +
	v \dbar (\xi R)^1.
\end{equation}
Moreover, from \eqref{eq:Dv2} it follows that
\begin{equation}
\begin{aligned}  \label{eq:DvxiR2}
	Dv(\dbar v)^j (\xi R) &=
	(\dbar v)^j (\xi R) +
	v(\dbar v)^{j-1} \dbar (\xi R) -
	v(\dbar v)^j \dbar (\xi R) \\ &-
	(\dbar v)^j (\xi R)^j -
	v(\dbar v)^{j-1} \dbar (\xi R)^j +
	v(\dbar v)^j \dbar (\xi R)^{j+1}
\end{aligned}
\end{equation}
for $j \geq 1$, since
\[
	\left( \sum_{\ell=0}^j v (\dbar v)^j D  (\xi R)^\ell \right) =
	\left( \sum_{\ell=0}^j v (\dbar v)^j \dbar (\xi R)^{\ell+1} \right) =
	v (\dbar v)^j \dbar (\xi R)^{j+1}.
\]

Using \eqref{eq:DvxiR1} and \eqref{eq:DvxiR2}, it follows that we have a telescoping sum
\begin{equation} \label{eq:nablawxiR}
	\nabla(w (\xi R)) = \xi R - \sum_j (\dbar v)^j (\xi R)^j.
\end{equation}
In addition, since $\nabla \xi = 0$ and $\deg \xi = k$, it follows by \eqref{eq:Rdef} that
$\nabla ((-1)^k \xi U) = \xi - \xi R$, and combining this with \eqref{eq:nablawxiR},
we obtain \eqref{eq:nablaPotential}.

\end{proof}

\begin{remark} \label{rem:dbarSign}
	For reasons of consistency with \cite{AW1} and later works concerning residue currents, we use a
	minus sign in front of $\dbar$ in the double complex $E^{\bullet,\bullet}$ in the proof above,
	which causes us to have to also equip the Dolbeault complex $\Hom(\mathcal{F},\currs^{0,\bullet}(G))$ in Theorem~\ref{thm:main} with
	the non-standard differential $-\dbar$. If we would skip the minus sign in front of $\dbar$ in the double complex $E^{\bullet,\bullet}$,
	it would alter signs in the definition of $R$, and Theorem~\ref{thm:main} would hold if $\Hom(\mathcal{F},\currs^{0,\bullet}(G))$
	was equipped with the differential $\dbar$.
\end{remark}

\subsection{Sheaf $\Exts$ and residue currents}
We end this section with a brief discussion on how one can use residue current to go between different representations of the sheaf $\Exts^k(\mathcal{F},G)$. This case offers no novelty since this is basically just a matter of putting some of the results of \cite{And1} into a global context and make appropriate interpretations so that it fits into the framework of twisted resolutions.

As usual, we let $(F,a)$ be a twisting cochain associated with $\mathcal{F}$. We have that for each index $\alpha$,
$a$ induces a complex
\begin{equation}
\label{eq:Homscomplex}
	0 \to \Homs(F_\alpha^0,G) \to \Homs(F_\alpha^{-1},G) \to \dots,
\end{equation}
where a section $\xi_\alpha$ of $\Homs(F_\alpha^{-k},G)$ is mapped to $\xi_\alpha a_\alpha^0$. It is well known that over each $\opens_\alpha$, $\Exts^k(\mathcal{F},G)$ can be computed as the $k$th cohomology of the complex \eqref{eq:Homscomplex}. Analogous to the discussion following Definition~\ref{def:twisting-cochain}, the cohomology sheaves $\mathcal{H}^k(\Homs(F_\alpha^\bullet,G))$ can be glued together over each intersection $\opens_{\alpha \beta}$ via the isomorphism induced by $a_{\alpha \beta}^1$, i.e.,
\begin{equation} \label{eq:Extsalphabeta}
    [\xi_\alpha] \mapsto [\xi_\alpha a_{\alpha \beta}^1].
\end{equation}
For simplicity, let us in the following result identify $\Exts^k(\mathcal{F},G)$ with this sheaf.
\begin{proposition}
\label{prop:Exts}
Let $\mathcal{F}$ be a coherent $\holo_X$-module, and let $G$ be a holomorphic vector bundle. Let $(F,a)$ be a twisted resolution of $\mathcal{F}$, and let $R$ be the associated residue current. Then there is a well-defined isomorphism
\begin{equation} \label{eq:Extsiso}
	\Exts^k(\mathcal{F},G) \overset{\cong}{\to}
	\mathcal{H}^k(\Homs(\mathcal{F},\currs^{0,\bullet}(G))),
\end{equation}
which over each $\opens_\alpha$ is given by
\begin{equation} \label{eq:ExtsisoFormula}
	[\xi_\alpha] \mapsto [\xi_\alpha R_\alpha^0].
\end{equation}
\end{proposition}

\begin{proof}
From \cite{And1} it follows that for each $\alpha$ there is an isomorphism
\[
	\mathcal{H}^k(\Homs(F_\alpha^\bullet,G)) \overset{\cong}{\to}
	\mathcal{H}^k(\Homs(\mathcal{F},\currs^{0,\bullet}(G)))
\]
over $\opens_\alpha$ given by \eqref{eq:ExtsisoFormula}.
We note that when we move to another open set $\opens_\beta$, then the source of
\eqref{eq:Extsiso} is mapped via the  isomorphism \eqref{eq:Extsalphabeta}.
In the target of \eqref{eq:Extsiso}, the implicit isomorphism $\mathcal{F} \cong F^0_\alpha/(\im a^0_\alpha)$ on $\opens_\alpha$ is used,
and thus, when we move to another open set $\opens_\beta$, then $[\xi_\alpha R^0_\alpha]$ is mapped to $[\xi_\alpha R^0_\alpha a^1_{\alpha\beta}]$,
so to prove that the isomorphism given by \eqref{eq:ExtsisoFormula} is well-defined, we should prove that
\begin{equation} \label{eq:ExtsisoCompatible}
	[\xi_\alpha R_\alpha^0 a_{\alpha \beta}^1] =
	[\xi_\alpha a_{\alpha \beta}^1 R_\beta^0].
\end{equation}
By \eqref{eq:R1comparisonFormula} and the fact that $\xi_\alpha a^0_\alpha = 0$,
we have that
\[
	\xi_\alpha R_\alpha^0 a_{\alpha \beta}^1 -
	\xi_\alpha a_{\alpha \beta}^1 R_\beta^0 =
	\xi_\alpha a_\alpha^0 R_{\alpha \beta}^1 -
	\xi_\alpha \dbar  R_{\alpha \beta}^1 =
	\pm \dbar (\xi_\alpha R_{\alpha \beta}^1),
\]
and hence we conclude that \eqref{eq:ExtsisoCompatible} holds.
\end{proof}

\section{The cup product and Serre duality}
\label{sect:serre}
One context in which global Ext groups appear is in the context of Serre duality for coherent $\holo_X$-modules. For simplicity, we will in this section assume $X$ to be compact, but the discussion below also applies to non-compact $X$, provided one considers cohomology with compact support in one of the factors,
and impose appropriate Hausdorffness assumptions.
Let as before $\mathcal{F}$ be a coherent $\holo_X$-module.
Recall that Serre duality states that there is a perfect pairing
\begin{equation}
\label{eq:SerrePairing}
	H^{n-k}(X,\mathcal{F}) \times \Ext^k(\mathcal{F},\omega) \to \C.
\end{equation}
We now spell out how this pairing is realized in \cite{MalSerre}. We have that
\[
	\mathcal{F} \longrightarrow \mathcal{F} \otimes \forms^{0,0}
	\overset{\dbar}{\longrightarrow} \mathcal{F} \otimes \forms^{0,1}
	\overset{\dbar}{\longrightarrow} \dots
\]
is a fine resolution of $\mathcal{F}$. Thus $H^k(X,\mathcal{F}) \cong H^k(\Hom(\holo_X,\mathcal{F} \otimes \forms^{0,\bullet}))$.
Here we have used the fact that for a $\holo_X$-module $\mathcal{G}$, its global sections are given by $\Hom(\holo_X,\mathcal{G})$,
a representation of global sections that fits well with the discussion below.
We also have an isomorphism $H^n(\Hom(\holo_X,\currs^{n,\bullet})) \cong \C$ via integration of currents.
The pairing \eqref{eq:SerrePairing} is then realized as a pairing
\begin{equation} \label{eq:serrePairing1}
    H^{n-k}(\Hom(\holo_X,\forms^{0,\bullet} \otimes \mathcal{F})) \times
    H^k(\Hom(\mathcal{F},\currs^{n,\bullet})) \to
    H^n(\Hom(\holo_X,\currs^{n,\bullet})),
\end{equation}
which is induced by combining the composition $\Homs(\holo_X,\mathcal{F}) \times \Homs(\mathcal{F},\currs^{n,k}) \to \Homs(\holo_X,\currs^{n,k})$
with the wedge product $\forms^{0,n-k} \times \currs^{n,k} \to \currs^{n,n}$.
We will denote this pairing simply by $\wedge$.

On the other hand, it is also possible to realize the pairing \eqref{eq:SerrePairing} by representing the Ext groups as the cohomology of a twisted complex. Let $(F,a)$ be a twisted resolution of $\mathcal{F}$. Let $\omega$ denote the sheaf of holomorphic $n$-forms on $X$, and consider the following graded groups
\begin{equation}
\begin{aligned} \label{eq:ABCdef}
    B^{q,\ell} &\defeq \bigoplus_{p+r=\ell} C^p(\cover,\forms^{0,q}(\Homs^r(\holo_X,F))) \cong
                \bigoplus_{p+r=\ell} C^p(\cover,\Homs^r(\holo_X,\forms^{0,q}(F))), \\
    E^{q,\ell} &\defeq \bigoplus_{p+r=\ell} C^p(\cover,\currs^{0,q}(\Homs^r(F,\omega))) \cong
                \bigoplus_{p+r=\ell} C^p(\cover,\Homs^r(F,\currs^{n,q})), \\
    G^{q,\ell} &\defeq C^\ell(\cover,\currs^{0,q}(\Homs(\holo_X,\omega))) \cong
                C^\ell(\cover,\Homs(\holo_X,\currs^{n,q})).
\end{aligned}
\end{equation}
We may identify $\holo_X$ and $\omega$ with the corresponding twisting cochains concentrated in degree $0$.
In this way, $E^{q,\ell}$ and $G^{q,\ell}$ are groups of the form \eqref{eq:cechCurrentHom}.
In addition, $B^{q,\ell}$ may be viewed as a subgroup of groups of this form, by using the inclusion $\forms^{0,q} \subseteq \currs^{0,q}$.
Note that $B^{\bullet,\bullet}$, $E^{\bullet,\bullet}$, and $G^{\bullet,\bullet}$ are all double complexes with differentials $D$ and $-\dbar$.
Recall the calculations of $H^q(E^{\bullet,\ell})$, \eqref{eq:HqEql}, and $H^\ell(E^{q,\bullet})$, \eqref{eq:HlEql}, in the proof of Theorem~\ref{thm:main}.
Similar arguments yield similar expressions for the cohomology of $B^{\bullet,\bullet}$ and $G^{\bullet,\bullet}$,
where in the case of $B^{\bullet,\bullet}$, one may use that $\forms^{0,q}$ is stalkwise flat and fine.
Hence, we may apply \eqref{eq:doubleComplexCohomologyIsomorphism}, which yields isomorphisms
\begin{equation}
\begin{aligned} \label{eq:cechDeRham}
    H^k\left( \bigoplus_{p+r=\bullet} C^p(\cover,\Homs^r(\holo_X,F))\right) \overset{\cong}{\to} H^k( \Hom(\holo_X,\forms^{0,\bullet}(\mathcal{F}))), \\
    H^k\left( \bigoplus_{p+r=\bullet} C^p(\cover,\Homs^r(F,\omega))\right)  \overset{\cong}{\to} H^k(\Hom(\mathcal{F},\currs^{n,\bullet})), \\
    H^k(C^{\bullet}(\cover,\Homs(\holo_X,\omega))) \overset{\cong}{\to} H^k(\Hom(\holo_X,\currs^{n,\bullet})).
\end{aligned}
\end{equation}

Using the isomorphisms \eqref{eq:cechDeRham}, we can alternatively express \eqref{eq:serrePairing1} as a pairing
\begin{equation}
\begin{gathered} \label{eq:serrePairing3}
    H^{n-k}\left(\bigoplus_{p+r=\bullet} C^p(\cover,\Homs^r(\holo_X,F))\right)\times H^k\left( \bigoplus_{p+r=\bullet} C^p(\cover,\Homs^r(F,\omega))\right)
     \\ \to H^n(C^\bullet(\cover,\Hom(\holo_X,\omega))).
\end{gathered}
\end{equation}
We will now show that it is given by the following cup product. Let $x$ and $y$ be $D$-closed elements of $\bigoplus_{p+r=n-k} C^p(\cover,\Homs^r(\holo_X,F))$ and  $\bigoplus_{p+r=k} C^p(\cover,\Homs^r(F,\omega))$ respectively. We define $[x] \smile [y]$ as $[xy]$, where $xy$ is the product \eqref{eq:prod2}. It is easy to see that this is well defined since if, for example, $x = D \alpha$, then $xy = D(\alpha y)$, and hence $[xy] = 0$. So the product is independent of the choice of representatives.
\begin{proposition}
    Let $\Psi$ denote the isomorphisms in \eqref{eq:cechDeRham},
    and let
    \[
     [x] \in H^{n-k}\left(\bigoplus_{p+r=\bullet} C^p(\cover,\Homs^r(\holo_X,F))\right)
    \text{ and } [y] \in H^k\left( \bigoplus_{p+r=\bullet} C^p(\cover,\Homs^r(F,\omega))\right).
    \]
    Then
    \begin{equation}
        \Psi( [x] \smile [y] ) = \Psi([x]) \wedge \Psi([y]).
    \end{equation}
\end{proposition}
\begin{proof}
Let $x'$ and $y'$ be representatives of $\Psi([x])$ and $\Psi([y])$. In view of \eqref{eq:doubleComplexCohomologyIsomorphism} and \eqref{eq:doubleComplexExplicitIsomorphism}, there exist elements $\alpha$ and $\beta$ such that $x - x' = \nabla \alpha$ and $y - y' = \nabla \beta$.
Using this, one readily verifies that
\begin{equation*}
	x y - x' y' = \nabla((-1)^{\deg x} x \beta + \alpha y').
\end{equation*}
Thus $x y$ and $x' y'$ are $\nabla$-cohomologous and from this the statement follows.
\end{proof}

\end{document}